\documentclass[11pt]{amsart}

\usepackage{enumerate,url,amssymb,mathrsfs,upref,pdfsync}

\newtheorem{theorem}{Theorem}[section]
\newtheorem{lemma}[theorem]{Lemma}
\newtheorem*{claim0}{Claim A}
\newtheorem*{claima}{Claim B}
\newtheorem*{claimb}{Claim C}
\newtheorem{proposition}[theorem]{Proposition}
\newtheorem{conjecture}[theorem]{Conjecture}
\newtheorem*{proposition*}{Proposition}

\theoremstyle{definition}
\newtheorem{definition}[theorem]{Definition}
\newtheorem{example}[theorem]{Example}
\newtheorem{question}[theorem]{Question}
\newtheorem{corollary}[theorem]{Corollary}

\theoremstyle{remark}

\numberwithin{equation}{section}

\newcommand{\abs}[1]{\lvert#1\rvert}
\newcommand{\norm}[1]{\lVert#1\rVert}

\newcommand{\bow}[1]{\overset{{}_{\bowtie}}{#1}}
\newcommand{\A}{\mathbb{A}}
\newcommand{\B}{\mathbb{B}}
\newcommand{\C}{\mathbb{C}}

\newcommand{\CC}{C}
\newcommand{\DD}{\mathbb{D}}
\newcommand{\E}{\mathcal{E}}
\newcommand{\EE}{\mathsf{E}}
\newcommand{\G}{\mathcal{G}}

\newcommand{\p}{\mathbf{P}}
\newcommand{\R}{\mathbb{R}}
\newcommand{\T}{\mathbb{T}}

\newcommand{\s}{\mathbb{S}}
\newcommand{\X}{\mathbb{X}}
\newcommand{\Y}{\mathbb{Y}}
\newcommand{\x}{\mathfrak{X}}
\newcommand{\y}{\Upsilon}
\newcommand{\N}{\mathbb{N}}

\newcommand{\D}{\mathfrak{D}}

\newcommand{\Ho}{\mathsf{H}}

\newcommand{\cto}{\xrightarrow[]{{}_{c\delta}}}
\newcommand{\onto}{\overset{{}_{\textnormal{\tiny{onto}}}}{\longrightarrow}}
\DeclareMathOperator{\dist}{dist}
\DeclareMathOperator{\Mod}{Mod}
\DeclareMathOperator{\re}{Re}
\DeclareMathOperator{\im}{Im}
\DeclareMathOperator{\id}{id}
\DeclareMathOperator{\diam}{diam}
\DeclareMathOperator{\inte}{Int}

\def\XXint#1#2#3{{\setbox0=\hbox{$#1{#2#3}{\int}$}
\vcenter{\hbox{$#2#3$}}\kern-.5\wd0}}

\def\le{\leqslant}
\def\ge{\geqslant}

\begin{document}

\title[Existence of energy-minimal diffeomorphisms]{Existence of energy-minimal diffeomorphisms between  doubly connected domains}

\author[Iwaniec]{Tadeusz Iwaniec}
\address{Department of Mathematics, Syracuse University, Syracuse,
NY 13244, USA}
\email{tiwaniec@syr.edu}

\author[Koh]{Ngin-Tee Koh}
\address{Department of Mathematics, Syracuse University, Syracuse,
NY 13244, USA}
\email{nkoh@syr.edu}

\author[Kovalev]{\\Leonid V. Kovalev}
\address{Department of Mathematics, Syracuse University, Syracuse,
NY 13244, USA}
\email{lvkovale@syr.edu}

\author[Onninen]{Jani Onninen}
\address{Department of Mathematics, Syracuse University, Syracuse,
NY 13244, USA}
\email{jkonnine@syr.edu}

\thanks{Iwaniec was supported by the NSF grant DMS-0800416 and the Academy of Finland grant 1128331.
Koh was supported by the NSF grant DMS-0800416.
Kovalev was supported by the NSF grant DMS-0968756.
Onninen was supported by the NSF grant  DMS-1001620.
}




\begin{abstract}
The paper establishes the existence of homeomorphisms between two planar domains that minimize  the Dirichlet energy. 
 \smallskip

{\it Among all homeomorphisms $f \colon \Omega \onto \Omega^\ast$ between bounded doubly connected domains such that $\Mod \Omega \le \Mod \Omega^\ast$ there exists, unique up to conformal authomorphisms of $\Omega$, an energy-minimal diffeomorphism.}

 \smallskip
No boundary conditions are imposed on $f$. Although any energy-minimal diffeomorphism is harmonic, our results underline the major difference between the existence of harmonic diffeomorphisms and the existence of the energy-minimal diffeomorphisms.
The existence of globally invertible  energy-minimal mappings is of primary pursuit in the mathematical models of nonlinear elasticity and is also of interest in computer graphics.  
\end{abstract}


\maketitle
\tableofcontents
\section{Introduction}\label{intsec}

Throughout this text  $\Omega$ and $\Omega^*$ will be bounded domains in the complex plane $\C$. The \emph{Dirichlet energy}
 of a diffeomorphism $f\colon \Omega\onto\Omega^*$ is defined and denoted by
\begin{equation}\label{ener1}
\E[f]= \int_{\Omega}\abs{Df}^2 = 2 \int_{\Omega}\left(\abs{\partial f}^2 + \abs{\bar \partial f}^2\right)
\end{equation}
where $\abs{Df}$ is the Hilbert-Schmidt norm of the differential matrix of $f$.
The primary goal of this paper is to establish the existence
of a diffeomorphism $f\colon \Omega\onto\Omega^*$ of  smallest (finite) Dirichlet energy. The behavior of such
an \emph{energy-minimal} diffeomorphism $f$ resembles that of a conformal mapping. Indeed, a change of variables in~\eqref{ener1}
yields
\begin{equation}\label{ener2}
\E[f] = 2\int_{\Omega} J_f(z)\, dz + 4\int_{\Omega} \abs{\bar \partial f}^2\ge 2 \abs{\Omega^\ast}
\end{equation}
where $J_f$ stands for the Jacobian determinant and $\abs{\Omega^*}$ is the area of $\Omega^*$. A conformal mapping of $\Omega$ onto $\Omega^*$; that is,
a homeomorphic solution of the Cauchy-Riemann system $\bar\partial f=0$, would be an obvious choice for the minimizer of~\eqref{ener2}. Unfortunately,  for generic multiply connected domains there is no such mapping.
The existence of an energy-minimal diffeomorphism $f\colon \Omega\onto\Omega^*$ may be interpreted as saying that the
Cauchy-Riemann equation $\bar\partial f=0$ admits a diffeomorphic solution in the least squares sense, meaning that
$\|\bar \partial f\|_{L^2}$ assumes its minimum.
For this reason energy-minimal diffeomorphisms are known under the name {\it least squares conformal mappings}  in the computer graphics literature~\cite{LPRM,RvBBWRF}. They are also of great interest in the theory of nonlinear elasticity  due to  the principle of noninterpenetration of matter~\cite{Ba0, Sv}.

An energy-minimal diffeomorphism may fail to exist when a minimizing sequence collapses, at least partially, onto the boundary of $\Omega^*$. This phenomenon was observed in the papers~\cite{AIM,IO}  for  a pair of circular annuli. A related phenomenon occurs in free boundary problems for minimal graphs, where it is called {\it edge-creeping}~\cite{Ch, HN, Tu}. Since
the boundary of $\Omega^\ast$ plays a crucial role in the minimization of energy among diffeomorphisms $f \colon \Omega \onto \Omega^\ast$, our questions are essentially different from widely studied variational problems for mappings between Riemannian manifolds where the target is usually assumed to have no boundary~\cite{BCL,Jo,Job,Le}. We do not prescribe boundary values of $f$,  nor do we suppose that it has a continuous boundary extension.

Any energy-minimal diffeomorphism satisfies Laplace's equation, since one can perform first variations while preserving the diffeomorphism property. However, the existence of a harmonic diffeomorphism does not imply the existence of an energy-minimal one, see Example~\ref{ellip}. This is why our necessary condition for the existence of an energy-minimal diffeomorphism, Theorem~\ref{emexist}, is more restrictive than the corresponding result for harmonic diffeomorphisms
in~\cite{IKO4}.

As we have already pointed out,  energy-minimal diffeomorphisms for simply connected domains  are obtained from the Riemann mapping theorem. The doubly connected case, being next in the order of complexity, is the subject of our main result.

\begin{theorem} \label{mainexist}
Suppose that $\Omega$ and $\Omega^*$ are bounded doubly connected domains in $\C$ such that  $\Mod \Omega\le \Mod\Omega^*$.
Then there exists an energy-minimal diffeomorphism $f\colon \Omega\onto \Omega^*$, which is unique up to a conformal change of variables in $\Omega$.
\end{theorem}

Hereafter  $\Mod \Omega$ stands for the {\it conformal modulus} of $\Omega$. Any bounded doubly connected
domain $\Omega\subset\C$ is conformally equivalent to some \emph{circular annulus} $\{z \colon r< \abs{z}<R \}$ with $0\le r<R<  \infty$.
The ratio $R/r$, being independent of the choice of conformal equivalence, defines $\Mod \Omega:=\log R/r$. The conformal modulus is infinite  precisely when the bounded component of $\C \setminus \Omega$  degenerates to a point. We call such domain a {\it punctured domain}. Theorem~\ref{mainexist} has the following corollary.
 \begin{corollary}
For any bounded doubly connected domain $\Omega$ and any punctured domain $\Omega^\ast$ there exists an energy-minimal diffeomorphism $f\colon \Omega\onto \Omega^*$, which is unique up to a conformal change of variables in $\Omega$.
\end{corollary}

In the converse direction we show (Theorem~\ref{emexist}) that there exists no energy-minimal diffeomorphism when $\Mod \Omega^\ast \le \Phi (\Mod \Omega)$.
Here $\Phi \colon (0, \infty) \to (0, \infty)$ is a certain function asymptotically equal the identity at infinity, $\lim\limits_{t \to \infty} \Phi(t)/t=1$.
It is in this asymptotic sense that Theorem~\ref{mainexist} is sharp.
 It is rather surprising that our existence result for energy-minimal diffeomorphisms relies only on the conformal modulus of the target. Indeed, the energy minimization problem is  invariant only with respect to a conformal change of variable in the domain, not in the target.

Yet in other perspectives, the classical Teichm\"uller theory is concerned with the existence of quasiconformal mappings $g\colon \Omega^\ast \onto \Omega$ with smallest $L^\infty$-norm of the distortion function
\[K_g(w)= \frac{\abs{Dg(w)}^2}{2\, J_g(w)}, \qquad \mbox{ a.e. } w \in \Omega^\ast. \]
Analogous questions about $L^1$-norm of $K_g$ lead to minimization of the Dirichlet energy of the inverse mapping  via the transformation formula
\begin{equation}\label{identity}
\norm{K_g}_{L^1(\Omega^\ast)}= \E[f], \qquad \mbox{ where } f=g^{-1} \colon \Omega \onto \Omega^\ast
\end{equation}
For rigorous statements let us recall that a homeomorphism $g\colon \Omega^\ast \onto \Omega$ of Sobolev class $W^{1,1}(\Omega^\ast)$ has integrable distortion if
\begin{equation}\label{finitedist}
\abs{Dg(w)}^2 \le 2 K(w)\, J_g(w) \qquad \mbox{ a.e. in } \Omega^\ast
\end{equation}
for some $K\in L^1(\Omega^\ast)$. The smallest such $K \colon \Omega^\ast \to [1, \infty)$, denoted by $K_g$, is referred to as the {distortion function} of $g$.

It turns out that the inverse of any mapping with integrable distortion has finite Dirichlet energy and the identity~\eqref{identity} holds. As a consequence of Theorem~\ref{mainexist} we obtain the following result.
\begin{theorem}\label{thmintdist}
Let $\Omega$ and $\Omega^\ast$ be bounded doubly connected domains in $\C$ such that $\Mod \Omega \le \Mod \Omega^\ast$. Among all homeomorphisms $g \colon \Omega^\ast \onto \Omega$ there exists, unique up to a conformal automorphism of $\Omega$, mapping of smallest $L^1$-norm of the distortion.
\end{theorem}

We conclude this introduction with a strategy of the proof of Theorem~\ref{mainexist}. The natural setting for our minimization problem is the Sobolev space $W^{1,2}(\Omega)$. In this paper  functions in the  Sobolev spaces are complex-valued.  Let us reserve the notation $\Ho^{1,2}(\Omega,\Omega^*)$ for the set of
all sense-preserving  $W^{1,2}$-homeomorphisms $h\colon \Omega\onto\Omega^*$. When this set is nonempty, we define
\begin{equation}\label{en2}
\EE_\Ho(\Omega,\Omega^*)= \inf \{\E[h]\colon h\in \Ho^{1,2}(\Omega,\Omega^*)\}.
\end{equation}
By virtue of the density of diffeomorphisms in $\Ho^{1,2}(\Omega,\Omega^*)$, see~\cite{IKOh}, the minimization of energy among sense-preserving diffeomorphisms leads to the same value $\EE_\Ho(\Omega,\Omega^*)$.
A homeomorphism $h\in \Ho^{1,2}(\Omega,\Omega^*)$ is \emph{energy-minimal} if it attains the infimum in~\eqref{en2}.
Let us emphasize that the set $\Ho^{1,2}(\Omega,\Omega^*)\subset W^{1,2}(\Omega)$ is unbounded.
Even bounded subsets of $\Ho^{1,2}(\Omega,\Omega^*)$ are lacking compactness, due to the loss of injectivity in passing to a limit
of homeomorphisms. One way out of this difficulty is to consider the weak closure of $\Ho^{1,2}(\Omega,\Omega^*)\cap \mathcal B$ where $\mathcal B$ is a sufficiently large ball in $W^{1,2}(\Omega)$ whose size depends only on $\EE_\Ho(\Omega,\Omega^*)$. This is the approach undertaken in~\cite{IOt,JS}. However, the presence of $\mathcal B$
creates problems of its own.  For instance, the resulting class of mappings is not closed under compositions with
self-diffeomorphisms of $\Omega$; inner variation of such mappings would be inadmissible.

That is why we introduce the class of so-called \emph{deformations}. These are
sense-preserving surjective mappings of the Sobolev class $W^{1,2}$ that can be approximated by homeomorphisms in a certain way. The precise definition is given in \S\ref{defsec}.
A deformation is not necessarily injective. In addition, an energy-minimal deformation need not be harmonic, since one cannot perform first variations $f+\epsilon \varphi$ within the class of deformations. This is why we rely on inner variations, which yield that the Hopf differential (\S\ref{hopsec}) of an energy-minimal deformation is holomorphic in $\Omega$ and real on its boundary.
We gain additional information about the Hopf differential from the Reich-Walczak-type inequalities (\S\ref{reisec}) which is where the conformal moduli of $\Omega$ and $\Omega^*$ enter the stage.

The  crucial idea of the proof of Theorem~\ref{mainexist} is to consider a one-parameter family of variational problems in which $\Omega$ changes continuously while $\Omega^\ast$ remains fixed. We establish
strict monotonicity of the minimal energy as a function of the conformal modulus of $\Omega$ (\S\ref{monsec}). The proof of Theorem~\ref{mainexist} together with its more refined variant, Theorem~\ref{q4},  is completed in~\S\ref{exisec}.  The proof of the nonexistence theorem, Theorem~\ref{emexist}, is presented in~\S\ref{nonsec}.
The interested reader is invited to look upon the open questions collected in~\S\ref{quesec}.

\section{Statements}\label{stasec}

A homeomorphism of a planar domain is either sense-preserving or sense-reversing. For  homeomorphisms of the Sobolev class $W^{1,1}_{\rm loc}(\Omega)$
this implies that the Jacobian determinant does not change sign: it is either
nonnegative or nonpositive at almost every point~\cite[Theorem 3.3.4]{AIMb}, see also~\cite{HM}. The homeomorphisms considered in this paper are sense-preserving unless stated otherwise.

Let $\Omega$ and $\Omega^*$ be bounded domains in $\C$. To every mapping $f \colon \Omega \to \overline{\Omega^\ast}$ we
associate a boundary distance function $\delta_f(z)=\dist (f(z), \partial \Omega^\ast)$ which is set to $0$ on the boundary of $\Omega$.

The following concept, which interpolates between $c$-uniform (i.e., uniform on compact subsets) and uniform convergence, proves to be effective.
\begin{definition}
A sequence of mappings $h_j\colon \Omega\to \overline{\Omega^*}$ is said to converge \emph{$c\delta$-uniformly} to $h\colon \Omega\to \overline{\Omega^*}$
if
\begin{itemize}
\item $h_j\to h$ uniformly on compact subsets of $\Omega$ and
\item $\delta_{h_j} \to  \delta_h$ uniformly on $\overline{\Omega}$.
\end{itemize}
We designate it as $h_j\cto h$.
\end{definition}

\begin{definition}\label{defdef}
A mapping $h\colon \Omega\to\overline{\Omega^*}$ is called a \emph{deformation} if
\begin{itemize}
\item\label{dd1} $h\in W^{1,2}(\Omega)$;
\item\label{dd2} The Jacobian $J_h:=\det Dh$ is nonnegative  a.e. in $\Omega$;
\item\label{dd2p} $\int_\Omega J_h \le \abs{\Omega^\ast}$;
\item\label{dd3} there exist sense-preserving homeomorphisms $h_j\colon \Omega\onto\Omega^*$, called an \emph{approximating sequence},  such that $h_j\cto h$  on $\Omega$.
\end{itemize}
The set of deformations $h\colon \Omega \to \overline{\Omega^\ast}$ is denoted by $\D(\Omega,\Omega^*)$.
\end{definition}

The first thing to note is $\Ho^{1,2}(\Omega,\Omega^*)\subset \D(\Omega,\Omega^*)$.
Outside of some degenerate cases, the set of deformations is nonempty by Lemma~\ref{laterr}
and is  closed under weak limits in $W^{1,2}(\Omega)$ by Lemma~\ref{wclosed}.
Define
\begin{equation}\label{en1}
\EE(\Omega,\Omega^*)= \inf \{\E[h]\colon h\in \D(\Omega,\Omega^*)\}
\end{equation}
where $\E[h]$ is as in~\eqref{ener1}.
A deformation that attains the infimum in~\eqref{en1} is called \emph{energy-minimal}.
It is obvious that $\EE_\Ho(\Omega,\Omega^*)\ge \EE(\Omega,\Omega^*)$, but whether the equality holds is not clear. We are now in the position to state the existence Theorem~\ref{mainexist} more precisely.

\begin{theorem}\label{q4}
Suppose that $\Omega$ and $\Omega^*$ are bounded doubly connected domains in $\C$ such that $\Mod \Omega\le \Mod\Omega^*$.
There exists a diffeomorphism $h\in \Ho^{1,2}(\Omega,\Omega^*)$ that minimizes the energy among
all deformations; that is, $\E[h]=\EE(\Omega,\Omega^*)$ and hence, $\EE_\Ho(\Omega,\Omega^*) = \EE(\Omega,\Omega^*)$. Moreover, $h$ is unique up to a conformal automorphism of $\Omega$.
\end{theorem}

In opposite direction, for every $\epsilon>0$ there exists a pair of smooth bounded doubly connected domains $\Omega,\Omega^*$ with $\Mod\Omega^*>\log \cosh \Mod \Omega-\epsilon$, for which  there is no energy-minimal homeomorphism in $\Ho^{1,2}(\Omega,\Omega^*)$. See~\cite[Corollary 3]{AIM} or Example~\ref{ellip}.
More generally, we have the following counterpart to Theorem~\ref{q4}.

\begin{theorem}\label{emexist}
There is a nondecreasing function $\Upsilon\colon (0,\infty)\to (0,1)$ such that $\lim\limits_{\tau\to\infty}\Upsilon(\tau)= 1$ and the following holds.
Whenever two bounded doubly connected domains  $\Omega$ and $\Omega^*$ in $\C$ admit an energy-minimal diffeomorphism $h\colon \Omega\onto\Omega^*$,
we have
\begin{equation}\label{eme1}
\Mod\Omega^*\ge (\Mod\Omega)\cdot \Upsilon(\Mod\Omega).
\end{equation}
Specifically, one can take
\begin{equation}\label{specific}
\begin{split}
\Upsilon(\tau)=\exp\left(-\frac{\pi^2}{2\tau}\right)\cdot& \Lambda\left(\coth \frac{\pi^2}{2\tau}\right),\qquad \quad  \qquad \qquad
{\rm where }\\  &\Lambda(t)= \frac{\log t-\log(1+\log t)}{2+\log t},\quad  \ t\ge 1.
\end{split}
\end{equation}
\end{theorem}

In~\S\ref{quesec} we conjecture that~\eqref{eme1} can be specified as $\Mod\Omega^*\ge \log\cosh\Mod\Omega$, which would be the sharp bound, known to be true for circular annuli~\cite{AIM}.

\section{Basic properties of deformations}\label{defsec}

In this section we establish the essential properties of the class of deformations $\D(\Omega,\Omega^\ast)$ introduced in Definition~\ref{defdef}. Among them is that $\D(\Omega,\Omega^\ast)$ is sequentially weakly closed and its members satisfy a change of variable formula~\eqref{useful}.

Deformations enjoy two distinct properties, both of which are commonly known in literature as {monotonicity}.
The topological   monotonicity is the subject of Lemma~\ref{mono}. To avoid confusion, in the following definition we use the term \emph{oscillation property}.

\begin{definition} \label{diamdefin}
Let $U$ be an open subset of $\C$. A continuous function $f\colon U\to\C$ is said to have \emph{oscillation property} if for every compact set $K\subset U$ we have
\begin{equation}\label{diamdef}
\diam f(K) = \diam f(\partial K).
\end{equation}
Note that for real-valued functions~\eqref{diamdef} can be stated as
\[
\min_{K}f=\min_{\partial K} f \le \max_{\partial K} f = \max_{K}f.
\]
\end{definition}

The relevance of this property to Sobolev mappings hinges on the following continuity estimate. If $f\in W^{1,2}(U)$ has the oscillation property, then
\begin{equation}\label{continuity}
\abs{f(z_1)-f(z_2)}^2 \le \frac{C\int_{2D}\abs{D f}^2 }{\log \big(e+\frac{\diam D}{\abs{z_1-z_2}}\big)},
\end{equation}
with $z_1,z_2\in D$ for a pair of concentric disks $D\subset 2D\subset U$. The constant $C$ is universal.
See, e.g., Corollary 7.5.1~\cite{IMb}.

The oscillation property~\eqref{diamdef} obviously holds for all homeomorphisms and is preserved under $c$-uniform limits. Therefore, deformations satisfy~\eqref{diamdef} and consequently~\eqref{continuity}: the local modulus
of continuity of a deformation is controlled by its energy.

Our approach to energy-minimal deformations involves the comparison of energies of
$h$ and $h\circ f$, where $f$ is a diffeomorphism or (more generally) a quasiconformal homeomorphism~\cite{Ahb, AIMb, LVb}.
It is important to observe that $h\circ f$ is also a deformation.

\begin{lemma}\label{KM}
Let $\Omega$, $\Omega^\ast$ and $\Omega_\circ$ be bounded domains in $\C$.
If $f\colon \Omega_\circ \onto \Omega$ is a quasiconformal mapping  then for any $h\in \D(\Omega, \Omega^\ast)$ we have $h\circ f \in \D(\Omega_\circ, \Omega^\ast)$.
\end{lemma}
\begin{proof} Since a $K$-quasiconformal mapping distorts the Dirichlet
integral only by a factor up to $K$, it follows that $h\circ f\in W^{1,2}(\Omega)$.
That the Jacobian of $h \circ f$ is nonnegative follows from the chain rule $\det D(h\circ f) = (\det Dh)(\det Df)$.
Finally, observe that if $h_j\cto h$   in $\Omega$, then $h_j\circ f\cto h\circ f$   in $\Omega_\circ$; this purely topological fact only requires $f$ to be  a homeomorphism.
\end{proof}

Let us recall a change of variable formula for Sobolev mappings, found in~\cite[Corollary 3.3.6]{AIMb}, \cite[Theorem 6.3.2]{IMb} and \cite{Ha}.

\begin{lemma}\label{cvdeform}
Let $\Omega$ and $\Omega^\ast$ be bounded domains in $\C$.
Suppose that $h\colon \Omega\to  \overline{\Omega^\ast}$ is continuous and belongs to $W^{1,1}_{\rm loc} (\Omega)$. Then for any measurable function $v \colon  \overline{\Omega^\ast} \to [0, \infty)$ we have
\begin{equation}\label{change}
\int_{\Omega} v\big(h(z)\big)\,  \abs{J_h(z)}\, dz \le \int_{\C} v(w) N_\Omega (h, w)\, dw.
\end{equation}
where $N_\Omega(h, w)$ is the cardinality of the preimage $h^{-1}(w)$. If, in addition, $h$ satisfies  Lusin's condition $(N)$  then the equality holds in~\eqref{change}.
\end{lemma}

Lusin's condition $(N)$ means that $\abs{f(E)}=0$ whenever $\abs{E}=0$.

Hereafter $\deg_\Omega (h, w)$ stands for the degree of a mapping $h$ with respect to a point $w$~\cite{Llb}.
The degree is well-defined provided that $h\in C(\overline{\Omega})$ and $w\notin h(\partial\Omega)$.
However, we work with mappings that are not necessarily continuous up to the boundary.
In that case $\deg_\Omega (h, w)$ still makes sense as long as the values of $h$ near $\partial \Omega$
are bounded away from $w$. Specifically, $\deg_\Omega (h, w) := \deg_{\widetilde \Omega} (h, w)$
where $\widetilde{\Omega}\Subset \Omega$ is any compactly contained domain such that
$\inf\limits_{\Omega\setminus \widetilde{\Omega}}\abs{h-w}>0$.

\begin{lemma}\label{surjective}
For any $h\in \D(\Omega, \Omega^\ast)$ we have  $h(\Omega) \supset \Omega^\ast$.
\end{lemma}

\begin{proof}
We will prove the stronger statement
\begin{equation}\label{deg}
\deg_\Omega (h, w)=1 \qquad  \mbox{for all } w\in \Omega^\ast.
\end{equation}
Pick a point $w\in \Omega^\ast$  and let $\delta= \dist(w, \partial \Omega^\ast)$. Consider the open set \[U= \Big\{z \in \Omega \colon \delta_h(z) > \frac{\delta}{4} \Big\} \Subset \Omega .\]
Let $\{ h_j\}$ be an approximating sequence for $h$.
For sufficiently large $j$ we have $\abs{\delta_{h_j} -\delta_h }< \frac{\delta}{4}$ in $\Omega$ and $\abs{h-h_j} \le \frac{\delta}{4}$ in $\overline{U}$. For all $z\in \Omega \setminus U$, $\delta_{h_j}(z) \le \delta/2$, hence
\begin{equation}\label{ideaaboutz}
\abs{h_j(z)-w} \ge \frac{\delta}{2} \qquad \mbox{for } z\in \Omega \setminus U.
\end{equation}
Since $h_j$ is a homeomorphism it attains the value $w$ at some point $z_\circ$ in $U$. Let $U_\circ$ be the component of $z_\circ$ in $U$. Clearly, $\deg_{U_\circ}(h_j, w)=1$. On the boundary $\partial U_\circ$ we have $\abs{h_j-h} \le \delta /4$, which together with~\eqref{ideaaboutz} imply
\[\deg_{U_\circ}(h, w) = \deg_{U_\circ}(h_j, w)=1.\]
It remains to observe that $h(z)\ne w$ for $z\in\Omega\setminus U_\circ$. Indeed, by~\eqref{ideaaboutz}
the preimage of the open disk $D(w,\delta /2 )$  under the homomorphism $h_j$ is a connected subset of $U$, hence
a subset  of $U_\circ$. It follows that
\[
\abs{h(z)-w}\ge \abs{h_j(z)-w}-\frac{\delta}{4}\ge \frac{\delta}{4},\qquad z\in \Omega\setminus U_\circ
\]
as desired.
\end{proof}

\begin{definition}
A continuous mapping $f\colon \mathbb X \to \mathbb Y$ between metric spaces $\mathbb X$ and $\mathbb Y$  is {\it monotone} if for each $y\in f(\X)$ the set $f^{-1}(y)$ is compact and connected.
\end{definition}

\begin{proposition}\label{why}~\cite[VIII.2.2]{Wh}
If $\X$ is compact and $f\colon \mathbb X \onto \mathbb Y$ is monotone then $f^{-1}(C)$ is connected for every connected set $C \subset \mathbb Y$.
\end{proposition}

See~\cite{Mc, Rab, Wh} for the background on monotone mappings. Deformations are closely related to monotone mappings of $\mathbb S^2$ onto itself. Given two  $k$-connected bounded domains $\Omega$ and $\Omega^\ast$  in $\C$, we choose and fix  homeomorphisms
\begin{equation}\label{kai1}
\chi \colon \Omega \onto \mathbb S^2 \setminus P \quad \mbox{and} \quad \chi^\ast \colon \Omega^\ast \onto \mathbb S^2 \setminus P
\end{equation}
where $P \subset \mathbb S^2$ consists of $k$ points referred to as punctures. A homeomorphism $h\colon \Omega \onto \Omega^\ast$ induces unique  homeomorphism $\bow{h} \colon \mathbb S^2 \onto \s^2$ such that
\begin{equation}\label{kai2}
\bow{h} \circ \chi=\chi^\ast \circ h.
\end{equation}

Note that $\bow{h}$ takes punctures into punctures in a one-to-one correspondence, though it may permute the elements of $P$. We claim that if a sequence of homeomorphisms $h_j \colon \Omega \onto \Omega^\ast$ converges $c\delta$-uniformly, then the mappings $\bow{h}_j$ converge uniformly on $S^2$. Indeed,
fix a small $\epsilon>0$ such that the $\epsilon$-neighborhood of the punctures, denoted $N_i(\epsilon)$, $i=1,\dots,k$, are disjoint. The uniform convergence of $\{\delta_{h_j}\}$ allows us to choose $\sigma>0$ such that each neighborhood $N_i(\sigma)$ is mapped by $\bow{h}_j$ into the union $\bigcup_i N_i(\epsilon)$ when $j$ is large.
Being connected, the set $\bow{h}_j(N_i(\sigma))$ must be contained in $N_{\pi(i)}(\epsilon)$ where $\pi$ is a permutation of the set $\{1,\dots,k\}$, possibly dependent on $j$. But in fact, $\pi$ does not depend on $j$ when $j$ is large enough, due to uniform convergence of $\bow{h}_j$ on the boundaries of $N_i(\sigma)$.

Thus we conclude that the sequence $\{\bow{h}_j\}$ converges uniformly to a surjective mapping, denoted by $\bow{h}$, which
leaves the set $P$ invariant but may permute its elements.
In a summary, for any deformation $h\in \D(\Omega, \Omega^\ast)$ there exists a unique mapping  $\bow{h} \colon \mathbb S^2 \onto \s^2$ satisfying~\eqref{kai2}. Being a uniform limit of self-homeomorphisms of $\s^2$, this mapping is monotone~\cite[IX.3.11]{Wh}.
The monotonicity of $\bow{h}$ has direct implications for $h$, which we state as a lemma for future references.

\begin{lemma}\label{mono}
Let $\Omega$ and $\Omega^\ast$ be bounded $k$-connected domains in $\C$, $k=1,2,\dots$ and  $h\in \D(\Omega, \Omega^\ast)$.  Then $\bow{h}$ is monotone. Consequently, for any  connected set $C\subset \Omega^\ast$ the preimage $h^{-1}(C)$ is also connected. Moreover, for every continuum $C \subset \overline{\Omega^\ast}$ the set $h^{-1}(C) \cup \Gamma_1 \cup \dots \cup \Gamma_\ell$ is a continuum, where $\Gamma_1, \dots, \Gamma_\ell$ are those selected components of $\partial \Omega$ which intersect the closure of $h^{-1}(C)$.
\end{lemma}

Concerning Lemma~\ref{mono} we remark that $h\colon \Omega\to\overline{\Omega^*}$ is not necessarily monotone; however, the restriction of $h$ to $h^{-1}(\Omega^*)$ is.

Next we turn to analytic properties of deformations.

\begin{lemma}\label{multilemma}
Let $\Omega$ and $\Omega^\ast$ be bounded domains in $\C$. If $h\in \D(\Omega, \Omega^\ast)$, then $h$ satisfies  Lusin's condition $(N)$ and  $N_\Omega(h, w) =1$ for almost every $w\in \Omega^\ast$.  Also $J_h=0$ almost everywhere in $\Omega \setminus h^{-1}(\Omega^\ast)$.
\end{lemma}

\begin{proof}
By Theorem A in~\cite{MaMa}  Lusin's condition $(N)$ is true for all  continuous $W^{1,2}$-mappings that satisfy the oscillation inequality~\eqref{diamdef}. Since the latter holds for any deformation (Lemma~\ref{noname}), the condition $(N)$ is satisfied.

By the definition of a deformation,
\[ \int_{\Omega}  {J_h(z)}\, dz \le \abs{\Omega^\ast}\]
Invoking Lemmas~\ref{cvdeform} and~\ref{surjective} we arrive at
\begin{equation}\label{eq16687}
\abs{\Omega^\ast} \ge \int_{\Omega}  {J_h(z)}\, dz = \int_{\C} N_\Omega (h,  w)\, dw \ge
  \int_{\Omega^\ast} N_\Omega (h, w)\, dw \ge \abs{\Omega^\ast}.
 \end{equation}
Therefore, equality holds throughout in~\eqref{eq16687}. This yields $N_\Omega (h,w)=1$ a.e. in $\Omega^\ast$ and $J_h=0$ a.e. in  $\Omega \setminus h^{-1}(\Omega^\ast)$, as claimed.
\end{proof}

\begin{corollary}\label{useless}
Let $\Omega$ and $\Omega^\ast$ be bounded domains in $\C$. If $h\in \D(\Omega, \Omega^\ast)$ and  $v \colon  \overline{\Omega^\ast} \to [0, \infty)$ is measurable, then
\begin{equation} \label{useful}
\int_\Omega v\big(h(z)\big) J_h(z)\, d z = \int_{\Omega^\ast} v(w)\, d w.
\end{equation}
\end{corollary}
\begin{proof}
Let $G=h^{-1}(\Omega^\ast)$.
Combining Lemmas~\ref{cvdeform} and~\ref{multilemma} we have
\[
\begin{split}
\int_\Omega v\big(h(z)\big) J_h(z)\, d z & = \int_{G} v\big(h(z)\big) J_h(z)\, d z = \int_{\C} v(w) N_G (h,w)\, dw \\
&= \int_{\Omega^\ast} v(w)\, d w.
 \qedhere  \end{split} \]
\end{proof}

In general, a deformation may take  a part of $\Omega$ into $\partial \Omega^\ast$. This is the subject of our next lemma.

\begin{lemma}\label{goodset}
Suppose that  $h\in \D(\Omega,\Omega^*)$  where $\Omega$ and $\Omega^\ast$ are bounded doubly connected domains.
Let $G= \{z\in \Omega \colon h(z) \in \Omega^\ast\}$. Then   $G$ is a domain  separating the boundary components of $\Omega$. Precisely, the two components of $\partial \Omega$ lie in different components of $\C \setminus G$.
\end{lemma}

\begin{proof} The set $G$ is open by the continuity of $h$, and connected by Lemma~\ref{mono}.
Let $\partial_I \Omega^*$ and $\partial_O \Omega^*$ be the inner and outer components of the boundary of $\Omega^*$.
The function
\[
\delta(z):=\frac{\dist(h(z),\partial_I \Omega^*)}{\dist(h(z),\partial_I \Omega^*)+\dist(h(z),\partial_O \Omega^*)},\qquad z\in\Omega,
\]
extends continuously  to $\mathbb C$ by setting the values $0$ and $1$ in the components of $\C \setminus \Omega$.
The disjoint open sets $\{z\in\C \colon \abs{\delta(z)} < 1/2\}$ and $\{z\in\C \colon \abs{\delta(z)} > 1/2\}$ cover $\C\setminus G$
in such a way that each of them contains one and only one boundary component of $\Omega$. Thus $G$ separates the components of $\partial \Omega$.
\end{proof}

In order to prove that $\D(\Omega,\Omega^*)$ is sequentially weakly closed, we need an  estimate
near the boundary stated as Proposition~\ref{trans} below. For Sobolev homeomorphisms
a similar result was proved in~\cite{IOt} in all dimensions. The extension beyond homeomorphisms is deferred to~\S\ref{appen}.

\begin{proposition}\label{trans}
Let $\Omega$ and $\Omega^\ast$ be bounded $k$-connected domains, $2\le k<\infty$. Denote their boundary components by
$\Gamma_i$ and $\Gamma^\ast_i$, $i=1,\dots,k$. Assume that $\diam \Gamma_i >0$ for all $1\le i\le k$. Then there exist functions $\eta_i$, $1\le i\le k$, continuous in $\C$ and vanishing on $\Gamma_i$, such that the following holds. If $h\colon \Omega\to\overline{\Omega^*}$ is a continuous $W^{1,2}$-mapping such that  $h(\Omega) \supset \Omega^\ast$,  $h$ is monotone on the set $h^{-1}(\Omega^\ast)$, and
\begin{equation}\label{bddtobdd}
\dist(h(z), \Gamma^\ast_i)\to 0 \quad \textrm{as}\quad  \dist(z, \Gamma_i)\to 0, \quad i=1,\dots, k,
\end{equation}
then
\begin{equation}\label{wehave}
\dist(h(z), \Gamma^\ast_i) \le \eta_i(z)\,  \sqrt{\E[h]}, \qquad i=1,\dots,k.
\end{equation}
 \end{proposition}

\begin{lemma}\label{noname}
Let $\Omega$ and $\Omega^\ast$ be bounded $k$-connected domains, $2\le k<\infty$.
Assume that the boundary components of $\Omega$ do not degenerate into points.
If a family of deformations $\mathcal F \subset \D(\Omega,\Omega^*)$ is bounded in $W^{1,2}(\Omega)$ then
it is equicontinuous on compact subsets of $\Omega$ and the family
$\Delta_{\mathcal{F}}:=\{\delta_h \colon h \in \mathcal F\}$ is equicontinuous on $\overline{\Omega}$.
\end{lemma}
\begin{proof}
The equicontinuity of $\mathcal F$ on compact subsets of $\Omega$ is readily seen from~\eqref{continuity}.
It follows that $\Delta_{\mathcal F}$ is equicontinuous on compact subsets as well.
To show that it is actually equicontinuous on $\overline{\Omega}$ it remains
to prove that for any $\epsilon>0$ there is a compact set $K\subset \Omega$ such that
$\delta_h<\epsilon$ on $\Omega\setminus K$ for all $h\in \mathcal F$. This is exactly what Proposition~\ref{trans}
delivers.
\end{proof}

\begin{lemma}\label{wclosed}
Let $\Omega$ and $\Omega^\ast$ be bounded $k$-connected planar domains, $2\le k < \infty$.
Assume that the boundary components of $\Omega$ do not degenerate into points.
If a sequence $\{h_j\}\subset \D(\Omega,\Omega^*)$ converges weakly in $W^{1,2}(\Omega)$, then
its limit belongs to $\D(\Omega,\Omega^*)$
\end{lemma}

\begin{proof}
Let $h$ be the weak limit of $h_j\in \D(\Omega,\Omega^*)$. Its Jacobian determinant $J_h$ is nonnegative a.e. in $\Omega$ due to $L^1$-weak convergence of Jacobians
under $W^{1,2}$-weak limits~\cite[Theorem 8.4.2]{IMb}. The weak convergence also implies that $\int_{\Omega} J_h \le \abs{\Omega^\ast}$. It remains to show that $h$ has an approximating sequence of homeomorphisms. For this it is enough to prove that $h_j \cto h$ in $\Omega$. Indeed, each $h_j$ being a $c\delta$-uniform limit of homeomorphisms, the diagonal selection will produce the desired approximating sequence.

By Lemma~\ref{noname} the sequence $\{h_j\}$ is equicontinuous on any compact subset of $\Omega$.
With the help of the Arzel\`a-Ascoli theorem it is routine to prove that $h_j \to h$ $c$-uniformly.
In particular, $\delta_{h_j}\to \delta_h$ pointwise. The convergence is uniform because
the functions $\delta_{h_j}$ are equicontinuous in $\overline{\Omega}$ by virtue of Lemma~\ref{noname}.
It follows that $h_j\cto h$ as claimed.
\end{proof}

Due to the weak lower semicontinuity of the Dirichlet energy, Lemma~\ref{wclosed} has a useful corollary.

\begin{corollary}\label{attain}
Under the hypotheses of Lemma~\ref{wclosed} there exists $h\in \D (\Omega, \Omega^\ast)$ such that $\E[h]= \EE (\Omega, \Omega^\ast)$.
\end{corollary}

Note that Lemma~\ref{wclosed} fails for $k=1$. Indeed, the M\"obius transformations
\[f_a(z)= \frac{z-a}{1-a\bar z}\]
converge weakly in $W^{1,2}$ to a constant mapping (not a deformation) as $a\to 1$.
We conclude this section with a promised remark on the existence  of homeomorphisms of class $\Ho^{1,2}(\Omega,\Omega^*)$.

\begin{lemma}\label{laterr}
Let $\Omega$ and $\Omega^*$ be bounded doubly connected domains in $\C$. Then $\Ho^{1,2}(\Omega,\Omega^*)$ is nonempty, except for one degenerate case when  $\Mod\Omega=\infty$ and $\Mod\Omega^*<\infty$. In this case there is no homeomorphism $h\colon \Omega \onto \Omega^\ast$ of Sobolev class $W^{1,2}(\Omega)$.
\end{lemma}

\begin{proof} Suppose that the degenerate case takes place. Then $\Omega = V\setminus\{z_0\}$ where $V$ is a simply connected domain. Since isolated points are removable for monotone $W^{1,2}$ functions~\cite[Theorem 3.1]{IOt}, the mapping $h$ has a continuous extension to $V$. But then $\Omega^*=h(V)\setminus \{h(z_0)\}$,
which contradicts the finiteness of $\Mod\Omega^*$.

Conversely, suppose that the degenerate case fails. If $\Mod\Omega=\Mod\Omega^*=\infty$, then there exists a conformal
mapping $h\colon \Omega\onto\Omega^*$ for which $\E[h]=2\abs{\Omega^*}<\infty$ by virtue of~\eqref{ener2}.
The remaining case is when both domains have finite modulus. Then we map them conformally onto circular annuli
$\A$ and $\A^*$ and compose them with a radial quasiconformal mapping $\psi \colon \A \onto \A^\ast$,
\[
\psi(z)=\abs{z}^{\alpha-1}z, \qquad \alpha=\frac{\Mod \Omega^*}{\Mod\Omega}.
\]
This creates an element in $\Ho^{1,2}(\Omega,\Omega^*)$.
\end{proof}

\section{Harmonic replacement}

Let $\Omega$ be a domain in $\C$ and $U \Subset \Omega$ a bounded simply connected domain.
For any continuous function $f\colon \Omega\to\C$ there exists a unique continuous function $\p_{U}f \colon \Omega\to\C$,
called the Poisson modification of $f$, such that $\p_{U}f$ is harmonic in $U$ and agrees with $f$ on $\Omega\setminus U$.
Indeed, the Dirichlet problem with continuous boundary data has a continuous solution in any simply connected domain, e.g.,~\cite[Theorem 4.2.1]{Ranb} or
~\cite[Ch.III]{GMb}. Furthermore, $\p_Uf\in W^{1,2}(\Omega)$ whenever $f\in W^{1,2}(\Omega)$. Although the latter fact is surely  known,
we give an explanation. The function $\p_Uf$ can be constructed by the Wiener method~\cite[Theorem III.5.1]{GMb} as a $c$-uniform limit
\begin{equation}\label{wiener}
\p_Uf = \lim_{n\to\infty} \p_{U_n}f,\qquad U_1\Subset U_2\Subset \dots
\end{equation}
where $\{U_n\}$ is an exhaustion of $U$ by smooth Jordan domains.
Since the difference $\p_{U_n}f-f$ vanishes on the smooth boundary $\partial U_n$, it extends by zero to a function in $W^{1,2}(\Omega)$.
Adding $f$ to it, we conclude that $\p_{U_n}f\in W^{1,2}(\Omega)$, with a uniform bound on the $W^{1,2}$-norm thanks to Dirichlet's principle.
Thus, $\{\p_{U_n}f\}$ contains a subsequence that converges weakly in $W^{1,2}(\Omega)$. Its limit must be $\p_Uf$
since $\p_{U_n}f\to \p_Uf$ uniformly.

The following lemma generalizes the well-known Rad\'o-Kneser-Choquet Theorem on the univalence of harmonic extensions. The added generality is in that the domain $U$ is not required to be Jordan.

\begin{lemma}[Modification Lemma]\label{RKC}
Let $U$ and $D$ be bounded simply connected domains in $\C$ with $D$ convex. Suppose that  $f$ is a homeomorphism from $U$ onto $D$ with  continuous extension $f\colon \overline{U} \to \overline{D}$. Then there exists a unique harmonic homeomorphism $h\colon U \onto D $ which agrees with $f$ on the boundary. Specifically,  $h$ has a continuous extension to $\overline{U}$ which coincides  with $f$ on $\partial U$.
\end{lemma}

\begin{proof}
The existence and uniqueness of a continuous harmonic extension $h$ of $f{\big|_{\partial U}}$ are well known. Also $h(U) \supset D$ by a straightforward degree argument and $h(U)\subset \overline{D}$  by the maximum principle. Thus the essence of the lemma is injectivity of $h$.

Let $\{D_n\}$ be an exhaustion of $D$ by convex domains and define $U_n = f^{-1}(D_n)$, which is a Jordan domain. By the Rad\'o-Kneser-Choquet Theorem~\cite[p. 29]{Dub} the mapping $h_n:=\p_{U_n}f$ is harmonic homeomorphism of $U_n$ onto $D_n$. As $n\to \infty$, $h_n \to \p_U f=:h$ $c$-uniformly on $U$, see~\cite[Ch.III]{GMb}.
The convergence of harmonic functions implies the convergence of  their derivatives. Therefore $J_{h_n}\to J_h$ pointwise, in particular  $J_{h} \ge 0$.
This means that the holomorphic functions $h_z$ and $\overline{h_{\bar z}}$ satisfy the inequality $\abs{\overline{h_{\bar z}}} \le \abs{h_z} $. This is only possible when either $\abs{\overline{h_{\bar z}}} < \abs{h_z} $ in $U$ or $\abs{\overline{h_{\bar z}}} \equiv \abs{h_z} $ in $U$. The second case cannot occur, for it would
yield $J_h\equiv 0$, contradicting $h(D)\supset U$.    Therefore $J_h>0$, so the mapping $h$ is a local diffeomorphism. Being also a $c$-uniform limit of homeomorphisms, $h$ is a diffeomorphism of $U$.
\end{proof}

We are now in the position to apply the Poisson modification to deformations.

\begin{lemma}\label{harmrep}
Let $\Omega$ and $\Omega^*$ be bounded $k$-connected domains, $1\le k< \infty$. Suppose that $h\in \D(\Omega,\Omega^*)$
 satisfies $h(\Omega)=\Omega^*$.
Let $D$ be a convex domain such that $\overline{D}\subset \Omega^*$. Denote $U=h^{-1}(D)$ and $g=\p_U h$.
Then
\begin{enumerate}[(i)]
\item $g \in \D(\Omega,\Omega^*)$
\item The restriction of $g$ to $U$ is a harmonic diffeomorphism onto $D$.
\item $\E[g]\le \E[h]$ with equality if and only if $g\equiv h$.
\end{enumerate}
\end{lemma}

\begin{proof} We use the notation of~\eqref{kai1} and~\eqref{kai2}. By Lemma~\ref{mono} the induced  mapping $\bow{h}\colon \s^2 \onto \s^2$ is monotone, so we can apply Theorem II.1.47 in~\cite{Rab}. According to which there exists a  monotone
mapping $f \colon \s^2 \onto \s^2$ which takes $\chi(U)$ homeomorphically onto $\chi^\ast(D)$ and agrees with $\bow{h}$ on $\s^2 \setminus \chi(U)$.  This allows us to apply the Modification Lemma~\ref{RKC} to $h$. Thus the Poisson modification $g=\p_Uh$ performs a harmonic diffeomorphism
of $U$ onto $D$. Clearly, $\bow{g}\colon \s^2 \onto \s^2$ is  monotone. Any such mapping can be uniformly approximated by homeomorphisms~\cite[Theorem II.1.57]{Rab}. We can actually alter slightly these homeomorphisms so as to obtain a sequence of homeomorphisms $g_j \colon \s^2 \onto \s^2$ that agree  with $\bow{g}$ at the punctures $P \subset \s^2$, and still  $g_j \to \bow{g}$ uniformly on $\s^2$.  Every such homeomorphism $g_j \colon \s^2 \onto \s^2$ is represented by a homeomorphism $h_j \colon \Omega \onto \Omega^\ast$ by the rule $g_j=\bow{h}_j $, where $\bow{h}_j$ is determined from the equation~\eqref{kai2}; $\bow{h}_j \circ \chi = \chi^\ast \circ h_j$. The uniform convergence of $g_j$ implies that $h_j \cto g$  in $\Omega$.   Thus we conclude that $g \in \D(\Omega,\Omega^*)$.
The inequality $\E[\p_U h]\le \E[h]$ is merely a restatement of Dirichlet's principle.
\end{proof}

\section{Reich-Walczak-type inequalities}\label{reisec}

The  Reich-Walczak inequalities~\cite{RW} provide the upper and lower bounds for the conformal modulus of the image of a circular annulus under a quasiconformal homeomorphism.
Propositions~\ref{rwrho} and~\ref{rwthe} provide such bounds for deformations, which are in general neither quasiconformal nor homeomorphisms. We also treat Sobolev homeomorphisms in $W^{1,1}_{\rm loc}$, for which
similar inequalities were established in~\cite{MM} in the context of self-homeomorphisms of a disk
that agree with the identity mapping on the boundary. However, we work with doubly connected domains and
do not prescribe boundary values.

Let us introduce notation for several quantities associated with the derivatives of a mapping $f$.
We make use of polar coordinates $\rho$ and $\theta$
and the associated {\it normal} and {\it tangential} derivatives
\[f_N=f_\rho \quad \mbox{ and } \quad f_T= \frac{f_\theta}{\rho}.\]
In these terms  the Wirtinger derivatives $f_z$ and $f_{\bar z}$ are expressed as
\[
f_z = \frac{e^{-i\theta}}{2}\left(f_N - i f_T\right) \qquad
f_{\bar z} = \frac{e^{i\theta}}{2}\left(f_N + i f_T\right)
\]
Also, the Jacobian determinant of $f$ is
\[
J_f= \abs{f_z}^2-\abs{f_{\bar z}}^2 = \im \overline{f_N} f_T.
\]
Except for the origin, where polar coordinates collapse, we may define the {\it normal} and {\it tangential distortion}  of $f$ as follows.
\begin{align}
K_N^f & := \frac{\abs{f_z+ \frac{\bar z}{z}f_{\bar z}}^2}{J_f}= \frac{\abs{f_N}^2}{ J_f}\\
K_T^f & := \frac{\abs{f_z- \frac{\bar z}{z}f_{\bar z}}^2}{J_f}= \frac{\abs{f_T}^2} {J_f}
\end{align}
By convention these quotients are understood as $0$ whenever the numerator vanishes. Naturally,
they assume the value $+\infty$ if the Jacobian vanishes but the numerator does not.
For a mapping $f\in W^{1,1}_{\rm loc}$ the quantities $f_N$, $f_T$, and $J_f$ are finite a.e.,
and therefore $K_N^f$ and $K_T^f$ are unambiguously defined at almost every point of the domain of
definition of $f$.

\begin{proposition}\label{rwrho}
Let $\Omega$ and $\Omega^\ast$ be bounded doubly connected domains such that $\Omega$ separates $0$ and $\infty$.
Suppose that \underline{either}
\begin{enumerate}[(a)]
\item $f\in \mathfrak{D}(\Omega, \Omega^*)$  \underline{or}
\item $f\colon\Omega \onto \Omega^*$ is a sense-preserving homeomorphism of class $W^{1,1}_{\rm loc}(\Omega, \Omega^*)$.
\end{enumerate}
Then
\begin{equation}\label{rwrho0}
2\pi \Mod\Omega^* \le \int_{\Omega} K_N^f \frac{d z}{\abs{z}^2}.
\end{equation}
\end{proposition}

\begin{proof}
There is nothing to prove if the integral in~\eqref{rwrho0} is infinite, so we assume $K_N^f <\infty$ a.e.
There exists a conformal mapping $\Phi \colon \Omega^\ast \to A(r_\ast,1)=: \A^\ast$ where $0 \le r_\ast<1$ is
such that $\Mod \Omega^\ast = \log 1/r_\ast$.
Let $G= \{z\in \Omega \colon f(z) \in \Omega^\ast\}$ and define $g \colon G \to \A^\ast$ by $g= \Phi \circ f$.
Note that $G=\Omega$ if we are in the case (b).

Fix $\epsilon >0$. We claim that
\begin{equation}\label{rhoderiv}
\int_G \frac{\abs{g_N}}{\abs{g} + \epsilon} \, \frac{dz}{\abs{z}} \ge 2 \pi \log \frac{1+\epsilon}{r_\ast+\epsilon}.
\end{equation}
Let $\ell_\theta= \{\rho e^{i\theta}\in G \colon \rho>0\}$, $\theta\in [0, 2\pi]$. For almost every $\theta \in [0, 2\pi]$ the mapping $g$ is locally absolutely continuous on  $\ell_\theta$. The image of $\ell_\theta$ under $g$ is a union of curves  in $\A^\ast$ which approach the boundary of $\A^\ast$ in both directions. At least one of them connects two boundary components of $\A^\ast$ because  $G$ separates the boundary components of $\Omega$ by Lemma~\ref{goodset}. Therefore the function $\abs{g}$ attains all values between $r_\ast$ and $1$ when restricted to some connected component of $\ell_\theta \cap G$. It follows that
\[
\int_{\ell_\theta \cap G} \frac{\abs{g_N}}{\abs{g} +\epsilon} \ge  \log \frac{1+\epsilon}{r_\ast+\epsilon}.
\]
Integration with respect to $\theta$ yields~\eqref{rhoderiv}.
Using the normal distortion inequality $ \abs{g_N}^2 \le K_N^g J_g$ and the
Cauchy-Schwarz inequality we obtain
\[
\begin{split}
\left( \int_G \frac{\abs{g_N}}{\abs{g} + \epsilon}\, \frac{dz}{\abs{z}} \right)^2 &\le
\left( \int_G \frac{ (K_N^g J_g)^{1/2}}{\abs{g} + \epsilon} \, \frac{dz}{\abs{z}} \right)^2\le
\int_G \frac{J_g}{(\abs{g} + \epsilon)^2} \; \int_G K_N^g \frac{d z}{\abs{z}^2} \\ &\le
\int_G \frac{J_g}{(\abs{g} + \epsilon)^2} \; \int_\Omega K_N^g\,\frac{d z}{\abs{z}^2} .
\end{split}
\]
Since $\Phi$ is conformal, $K_N^g=K_N^f$.  Thus we infer from~\eqref{rhoderiv} that
\begin{equation}\label{nothing}
\left( \log \frac{1+\epsilon}{r_\ast+\epsilon}  \right)^2 \le
\frac{1}{(2 \pi)^2}  \int_\Omega K_N^f\,\frac{ d z}{\abs{z}^2} \,
\int_G \frac{J_g}{(\abs{g} + \epsilon)^2}.
\end{equation}
From Lemmas~\ref{cvdeform} and~\ref{multilemma} we obtain
\begin{equation}\label{est9}
\int_{G} \frac{J_g}{(\abs{g} + \epsilon)^2} \le \int_{\A^\ast} \frac{dw}{(\abs{w}+\epsilon)^2} \le 2 \pi \log \frac{1+\epsilon}{r_\ast+\epsilon}.
\end{equation}
It follows from~\eqref{nothing} and~\eqref{est9}  that
\[ \log \frac{1+\epsilon}{r_\ast+\epsilon} \le \frac{1}{2\pi} \int_\Omega K_N^f\,\frac{ d z}{\abs{z}^2}. \]
Letting $\epsilon \to 0$ completes the proof.
\end{proof}

Unlike Proposition~\ref{rwrho}, our lower bound for the modulus of the image under a deformation
depends on  the rectifiability of the boundary of $\Omega^*$.
We do not know if this assumption is redundant.

\begin{proposition}\label{rwthe}
Let $\A=A(r,R)$ be a circular annulus, $0\le r<R<\infty$, and $\Omega^*$ a bounded doubly connected domain with finite modulus.
Suppose that \underline{either}
\begin{enumerate}[(a)]
\item $f\in \mathfrak{D}(\A,\Omega^*)$ and $\Omega^\ast$ is bounded by rectifiable Jordan curves, \underline{or}
\item $f\colon\A\onto \Omega^*$ is a sense-preserving homeomorphism of class $W^{1,1}_{\rm loc}(\A,\Omega^*)$.
\end{enumerate}
Then
\begin{equation}\label{rwthe0}
 \int_{\A} K_T^f \, \frac{d z}{\abs{z}^2} \ge 2\pi \frac{(\Mod \A)^2}{\Mod \Omega^\ast}.
\end{equation}
\end{proposition}

Before proving Proposition~\ref{rwthe} we collect some results concerning the Hardy space $H^1(\A)$ on a circular annulus $\A=A(r,R)$, $0<r<R<\infty$. In what follows $\T_\rho =\{z \in \C \colon \abs{z}=\rho\}$, $\rho > 0$.
By definition, a holomorphic function $\psi\colon \A\to \C$
belongs to $H^1(\A)$ if the integrals $\int_{\T_\rho} \abs{\psi}$ are uniformly bounded for $r<\rho<R$.
Such a function $\psi$ has nontangential limits a.e. on $\partial\A$~\cite[p.6]{Sa}, and $\psi\ne 0$ a.e. on $\partial A$
unless $\psi\equiv 0$~\cite[pp.10--12]{Sa}. The relation between $H^1(\A)$ and domains with rectifiable boundaries is summarized
in the following proposition which is a version of classical theorems due to F.~and M.~Riesz and V.~I.~Smirnov.
Below $\mathcal H^1$ denotes the one-dimensional Hausdorff measure, not to be confused with the Hardy space.

\begin{proposition}\label{hardy} Let $\Omega$ be a doubly connected domain  bounded by rectifiable Jordan curves and let
$\Psi \colon \A=A(r,R) \to \Omega$ be conformal. Then
\begin{enumerate}[(i)]
\item\label{nn} $\Psi'\in H^1 (\A)$
\item\label{nnn} for any Borel set $E\subset \partial \A$ we have $\mathcal H^1(\Psi(E)) = \int_E \abs{\Psi'}$
\item\label{nnnn} $\mathcal H^1(\Psi(E))=0$ if and only if $\mathcal H^1(E)=0$.
\end{enumerate}
\end{proposition}

\begin{proof} Part~\eqref{nn} is proved in exactly the same way as the corresponding result for the disk~\cite[p.~200]{GMb}.
Since $\Psi'\in H^1$, the continuous extension of $\Psi$ to $\partial \A$ is absolutely continuous, i.e.,~\eqref{nnn} holds.
Part~\eqref{nnnn} follows from~\eqref{nnn} because $\Psi'\ne 0$ a.e. on $\partial \A$.
\end{proof}

\begin{proof}[Proof of Proposition~\ref{rwthe}]
There is nothing to prove if the integral in~\eqref{rwrho0} is infinite, so we assume
$K_T^f <\infty$ a.e.
Let $G= \{z\in \A \colon f(z) \in \Omega^\ast\}$. Note that $G$ coincides with $\A$ if we are in the case~(b).
On the set $\A\setminus G$ the Jacobian $J_f$ vanishes by Lemma~\ref{goodset}.  Since
$K_T^f$ is finite a.e., it follows that $f_\theta=0$ a.e. on $\A\setminus G$.
There exists a conformal mapping $\Phi \colon \Omega^\ast \to A(r_\ast,1)=: \A^\ast$, where $0< r_\ast<1$
is determined by $\Mod \Omega^\ast = \log 1/r_\ast$.
In case (a) $\Phi$ extends to a homeomorphism $\Phi\colon \overline{\Omega^*}\to \overline{\A^*}$.
In either case (a) or (b) we can define $g= \Phi \circ f \colon \A \to \overline{\A^\ast}$.

We  claim that
\begin{equation}\label{thederiv}
\int_G \frac{\abs{g_T}}{ \abs{g} }\, \frac{dz}{\abs{z}} \ge 2 \pi \Mod \A.
\end{equation}
Indeed, for almost every circle $\T_\rho\subset \A$ the mapping $f$ is absolutely continuous on $\T_\rho$
and
\begin{equation}\label{cru1}
f_\theta=0 \quad \text{a.e. on }\T_\rho\setminus G.
\end{equation}
For any such $\rho$ we are going to prove the inequality
\begin{equation}\label{the2}
\int_{\T_\rho \cap G} \frac{\abs{g_T}}{\abs {g}} \ge 2 \pi,
\end{equation}
from which~\eqref{thederiv} will follow by integration.

In the case~(b) the inequality~\eqref{the2} is a direct consequence of the fact that  the curve
$g(\T_\rho)$ separates the boundary components of $\A^*$; indeed, the length of any such curve
in the logarithmic metric $\abs{dz}/\abs{z}$ is at least $2\pi$.

We now turn to the case~(a). Let $w_\circ$ be an interior point of the bounded component of $\C\setminus \Omega^*$.
Choose an approximating sequence $\{h_j\}_{j\in\N} \subset  {\Ho}^{1,2}(\Omega,\Omega^*)$ that converges to $f$.
Note that for each $j$ the multivalued function $\arg (h_j(z)-w_\circ)$ increases by $2\pi$ on $\T_\rho$.
Letting $j\to \infty$ we obtain the same for $f$; in particular, $f(\T_\rho)$ separates $w_\circ$ from $\infty$.
Since $\Phi\colon \overline{\Omega^*}\to \overline{\A^*}$ is a homeomorphism, $g(\T_\rho)$ is a closed curve in $\overline{\A^*}$ which separates $0$ from $\infty$. Therefore, its length in the logarithmic metric $\abs{dz}/\abs{z}$ is at least $2\pi$.
By virtue of~\eqref{cru1} the intersection of $f(\T_\rho)$ with $\partial\Omega^*$
has zero length. By part~\eqref{nnnn} of Theorem~\ref{hardy} we have
$\mathcal H^1(g(\T_\rho)\cap \partial \A^*)=0$.
Hence, the part of $g(\T_\rho)$ that is contained in $\A^*$ has logarithmic length at least $2\pi$. This is exactly what~\eqref{the2} claimed.

Now that~\eqref{thederiv} is at our disposal, we proceed as in the proof of Proposition~\ref{rwrho}.
The Cauchy-Schwarz inequality yields
\begin{equation}\label{star}
\begin{split}
\left( \int_G \frac{\abs{g_T}}{ \abs{g} } \, \frac{dz}{\abs{z}} \right)^2 &\le
\left( \int_G \frac{ (K_T^g J_g)^{1/2}}{  \abs{g} } \, \frac{dz}{\abs{z}} \right)^2\le
\int_G \frac{J_g}{\abs{g}^2} \; \int_G K_T^g\,  \frac{d z}{\abs{z}^2} \\
&\le  2\pi \log \frac{1}{r_*} \; \int_\A K_T^g\,\frac{d z}{\abs{z}^2}
= 2\pi \Mod\Omega^* \int_\A K_T^g\,\frac{d z}{\abs{z}^2} .
\end{split}
\end{equation}
where the second to last inequality follows from~\eqref{est9}. It remains to combine~\eqref{thederiv} and~\eqref{star}.
\end{proof}

\section{Hopf differentials}\label{hopsec}

We call a deformation $h\in \D(\Omega,\Omega^*)$ \emph{stationary}
if
\begin{equation}\label{stat}
\frac{d}{dt}\bigg|_{t=0}\E[h\circ \phi_t^{-1}]=0
\end{equation}
for every family of diffeomorphisms $\phi_t\colon \Omega\to\Omega$ which depend smoothly on the parameter $t\in\mathbb R$ and satisfy $\phi_0=\id$. It should be emphasized that apart from $\phi_0$, the diffeomorphisms $\phi_t$
need not agree with the identity on the boundary.
The derivative in~\eqref{stat} exists for any $h\in W^{1,2}(\Omega)$, see~\cite[p.~158]{SY}.
Every energy-minimal deformation is stationary. Indeed,
$h\circ \phi_t^{-1}$ belongs to $\D(\Omega,\Omega^*)$ by virtue of Lemma~\ref{KM}.
The minimal property of $h$ implies $\E[h\circ \phi_t^{-1}]\ge \E[h]$, from where ~\eqref{stat} follows.

The key property of the stationary mapping in~\eqref{stat} is that:
\begin{itemize}
\item The function $\varphi:= h_z\overline{h_{\bar z}}$, a priori in $L^1(\Omega)$, is actually  holomorphic.
\item If $\partial \Omega$ is $\CC^1$-smooth then $\varphi$ extends continuously to $\overline{\Omega}$, and the quadratic differential $\varphi \, dz^2$  is real on each boundary curve of $\Omega$.
\end{itemize}

See~\cite[Lemma 1.2.5]{Job} for the proof of the above facts and~\cite[Chapter~III]{Jeb} for the background on quadratic differentials and their
trajectories.
Let us consider the special case $\Omega=A(r,R)$ with $0<r<R<\infty$. Since $\varphi \, dz^2$
is real on each boundary circle, the function $z^2 \varphi(z)$ is real on $\partial \Omega$. By the maximum principle
\begin{equation}\label{above}
z^2\varphi(z)\equiv c\in\R.
\end{equation}
We state this as a lemma for the ease of future references.

\begin{lemma}\label{ctheory}
Let $\Omega=A(r,R)$ be a circular annulus, $0<r<R<\infty$, and $\Omega^*$ a bounded doubly connected domain.
If $h\in \D(\Omega,\Omega^*)$ is a stationary deformation, then
\begin{equation}\label{hopf1}
h_z\overline{h_{\bar z}} \equiv \frac{c}{z^2}\qquad \text{in }\Omega.
\end{equation}
where $c\in\R$ is a constant. Furthermore,
\begin{equation}\label{important}
\begin{cases}
\abs{h_N}^2 \le J_h, & \quad \mbox{if } c \le 0 \\
\abs{h_T}^2 \le J_h, & \quad \mbox{if } c \ge 0.
\end{cases}
\end{equation}
\end{lemma}

\begin{proof} The validity of~\eqref{hopf1} with some $c\in\R$ was already recognized in~\eqref{above}.
Separating the real and imaginary parts in~\eqref{hopf1} we arrive at two equations
\begin{align}
\abs{h_N}^2 - \abs{h_T}^2 &= \frac{4c}{\abs{z}^2}; \label{hopf2a}\\
\re (\overline{h_N} h_T) &=0. \label{hopf2b}
\end{align}
Recall that $J_h = \im \overline{h_N} h_T\ge 0$, which in view of~\eqref{hopf2b} reads as
\begin{equation}\label{hopf5}
J_h = \abs{h_N} \abs{h_T}
\end{equation}
Combining this with~\eqref{hopf2a} the claim~\eqref{important} follows.
\end{proof}

Lemma~\ref{ctheory} together with Propositions~\ref{rwrho} and~\ref{rwthe} give the following corollary.

\begin{corollary}\label{cpositive} Under the hypotheses of Lemma~\ref{ctheory}, we have
\begin{itemize}
\item if $\Mod\Omega<\Mod\Omega^*$, then $c>0$
\item if $\Mod\Omega>\Mod\Omega^*$ and $\Omega^*$ is bounded by
rectifiable Jordan curves, then $c<0$.
\end{itemize}
\end{corollary}

\section{Monotonicity of the minimum energy function}\label{monsec}

Due to the conformal invariance of the Dirichlet integral and of the class of deformations (Lemma~\ref{KM}), the minimal energy level
$\EE(\Omega,\Omega^*)$, defined by~\eqref{en1}, depends only on the conformal type of $\Omega$ as long
as $\Omega$ is bounded and $\Omega^\ast$ is fixed.
This leads us to consider a one-parameter family of extremal problems for homeomorphisms
$A(\tau)\onto\Omega^*$ of annuli $A(\tau)=A(1,e^{\tau})$, $0<\tau<\infty$.
In this section we are concerned with the quantity $\EE(\tau,\Omega^*):=\EE(A(\tau),\Omega^*)$ as a function of $\tau$, called the {\it minimum energy function}. When $\Omega^*$ has finite conformal modulus, $\EE(\tau,\Omega^*)$ attains its minimum at $\tau=\Mod\Omega^*$. Indeed, by~\eqref{ener2} for every $\tau$ we have $\EE (\tau, \Omega^\ast) \ge 2 \abs{\Omega^\ast}$, with equality if and only if $\Omega$ and $\Omega^\ast$ are conformally equivalent; that is, for $\tau= \Mod \Omega = \Mod \Omega^\ast$. The following monotonicity result, which extends this observation, will be of crucial importance
in the proof of Theorem~\ref{q4}.

\begin{proposition}\label{q3}
Let $\Omega^*$ be a bounded doubly connected domain.
The function $\tau\mapsto \EE(\tau,\Omega^*)$ is strictly decreasing
for $0<\tau<\Mod\Omega^*$. If, in addition, $\Omega^*$ is bounded by
rectifiable Jordan curves, then
$\EE(\tau,\Omega^*)$ is strictly increasing for $\tau>\Mod\Omega^*$.
\end{proposition}

The proof of Proposition~\ref{q3} requires auxiliary results concerning the normal and tangential  energies
\[\E_N[h]=\int_{\Omega}\abs{h_N}^2,\qquad  \E_T[h]=\int_{\Omega}\abs{h_T}^2.\]
Clearly $\E[h]=\E_N[h]+\E_T[h]$. Unlike $\E[h]$, both quantities $\E_N[h]$ and $\E_T[h]$
transform in a straighforward way under composition with the power stretch mapping
\begin{equation}\label{stret0}
\psi_{\alpha}(z):=\abs{z}^{\alpha-1}z, \qquad 0<\alpha<\infty.
\end{equation}
Specifically,
\begin{equation}\label{stret}
\E_N[h\circ \psi]= \alpha\, \E_N[h], \qquad \E_T[h\circ \psi]= \frac{1}{\alpha} \,\E_T[h].
\end{equation}
The direct verification of~\eqref{stret} is left to the reader. We only note that the domain of definition
of $h$ here is irrelevant as the computation is local.

\begin{lemma}\label{majorant}
Let $\Omega^*$ be a bounded doubly connected domain, $\tau_\circ\in (0,\infty)$.
Suppose that $h^\circ\in \D(A(\tau_\circ),\Omega^*)$ is an energy-minimal deformation.
Then for all $0<\tau<\infty$ we have
\begin{equation}
\EE(\tau,\Omega^*) \le \frac{\tau_\circ}{\tau}\,\E_N[h^\circ] + \frac{\tau}{\tau_\circ}\,\E_T[h^\circ].
\end{equation}
\end{lemma}

\begin{proof} Let $\alpha=\frac{\tau_\circ}{\tau}$ and note that $\psi_\alpha$ defined by~\eqref{stret0} is a quasiconformal mapping of  $A(\tau) $ onto $A(\tau_\circ)$.
By Lemma~\ref{KM} the composition $h^\circ \circ \psi_\alpha$ belongs to  $\D(A(\tau),\Omega^*)$
and by~\eqref{stret} we have
\[
\EE(\tau,\Omega^*) \le \E[h^\circ \circ \psi_\alpha] =
\alpha \,\E_N[h^\circ] + \frac{1}{\alpha}\, \E_T[h^\circ] . \qedhere
\]
\end{proof}

Let us apply Lemma~\ref{majorant} with $\tau_\circ = \Mod\Omega^*$. In this case $h^\circ \colon \Omega \onto \Omega^\ast$ is conformal so $\E_N[h^\circ]=\E_T[h^\circ]=\abs{\Omega^\ast}$. We obtain a simple
upper bound for the  minimal energy function,
\begin{equation}\label{upper}
\EE(\tau,\Omega^*) \le \left(\frac{\Mod\Omega^*}{\tau} + \frac{\tau}{\Mod\Omega^*}\right)
\abs{\Omega^*}, \quad 0< \tau < \infty.
\end{equation}

\begin{corollary}\label{econt}
The function $\EE(\tau,\Omega^*) $ is locally Lipschitz
for $0<\tau<\infty$.
\end{corollary}

Indeed the existence of $h^\circ$ in Lemma~\ref{majorant}  is assured by Corollary~\ref{attain}. From Lemma~\ref{majorant} for arbitrary $0 < \tau_\circ, \tau < \infty$ we have
\begin{equation}\label{econt1}
\begin{split}
\EE(\tau,\Omega^*) -\EE(\tau_\circ,\Omega^*)
& \le  \frac{\tau_\circ}{\tau} \E_N[h^\circ] + \frac{\tau}{\tau_\circ} \E_T[h^\circ] -  \E_N[h^\circ] - \E_T[h^\circ]\\
&=  (\tau-\tau_\circ)\left\{\frac{\E_T[h^\circ]}{\tau_\circ}-\frac{\E_N[h^\circ]}{\tau}\right\}
\end{split}
\end{equation}
from where the local Lipschitz property is readily seen.

\begin{proof}[Proof of Proposition~\ref{q3}]
Since $\EE(\tau,\Omega^*)$ is locally Lipschitz, its derivative
exists for almost every $\tau \in(0,\infty)$. Fix such a point of differentiablity, say  $0<\tau_\circ<\Mod\Omega^*$. Let $h^\circ \in \D(A(\tau_\circ),\Omega^*)$ be an energy-minimal
deformation. By Lemma~\ref{ctheory}
\[
\abs{h^\circ_N}^2 = \abs{h^\circ_T}^2 + \frac{4c}{\abs{z}^2}, \quad \mbox{hence upon integration } \E_N[h^\circ]=\E_T[h^\circ]+8c \pi \tau_\circ.
\]
Now, for any $\tau \in (0,\infty)$ the estimate~\eqref{econt1} takes the form
\begin{equation}
\EE(\tau,\Omega^*)-\EE(\tau_\circ,\Omega^*)
\le (\tau-\tau_\circ) \left\{-8 c\pi + (\tau_\circ^{-1}-\tau^{-1}) \E_N[h^\circ]\right\}
\end{equation}
Therefore
\[\frac{d}{dt}\bigg|_{\tau=\tau_\circ} \EE (\tau, \Omega^\ast)= -8 \pi c.\]
Corollary~\ref{cpositive} completes the proof.
\end{proof}

\section{Existence: Theorem~\ref{q4}}\label{exisec}

\begin{proposition}\label{gtheory}
Let $\Omega$ and $\Omega^*$ be bounded doubly connected domains.
Suppose that $h\in \D(\Omega,\Omega^*)$ satisfies $\E[h]=\EE(\Omega,\Omega^*)$.
Let $G=\{z\in\Omega\colon h(z)\in\Omega^*\}$. Then $G$ is a doubly connected domain that separates the boundary components
of $\Omega$. The restriction of $h$ to $G$ is a harmonic diffeomorphism onto $\Omega^*$.
\end{proposition}

\begin{proof} The fact that $G$ is a domain separating the boundary components of $\Omega$ was established in Lemma~\ref{goodset}.
Each point $z\in G$ has a neighborhood in which $h$ is a harmonic diffeomorphism.
Indeed, otherwise we would be able to find a deformation with strictly smaller energy by means of Lemma~\ref{harmrep}.
Thus, $h\colon G\onto\Omega^\ast$ is a local diffeomorphism.
On the other hand, for each $w\in \Omega^\ast$ the preimage $h^{-1}(w)$ is connected by Lemma~\ref{mono}. It follows that  $h\colon G\onto\Omega^*$ is a diffeomorphism.
Being a diffeomorphic image of $\Omega^*$, the domain $G$ must be doubly connected.
\end{proof}

\begin{proof}[Proof of Theorem~\ref{q4}]
If
$\Mod\Omega=\Mod\Omega^*$, then the domains are conformally equivalent. As observed in \S\ref{intsec}, a conformal mapping minimizes the Dirichlet energy. Thus we only need to consider the case
$\Mod\Omega<\Mod\Omega^*$. In particular $\Mod \Omega < \infty$.

Let $h$ and $G$ be as in Proposition~\ref{gtheory}. The existence of such $h$ is guaranteed by Corollary \ref{attain}. Since $G$ separates the boundary components of $\Omega$,
we have $\Mod G\le \Mod\Omega$ with equality if and only if $G=\Omega$~\cite[Lemma~6.3]{LVb}.
If $\Mod G< \Mod\Omega$, then by Proposition~\ref{q3}
\[
\int_G \abs{Dh}^2 \ge \EE(\Mod G,\Omega^*) > \EE(\Mod \Omega,\Omega^*) = \int_\Omega \abs{Dh}^2
\]
which is absurd. Thus $G=\Omega$.
By Proposition~\ref{gtheory} the mapping $h\colon \Omega\to\Omega^*$ is a harmonic diffeomorphism. The uniqueness statement will follow from Proposition~\ref{needed}.
\end{proof}
\begin{proof}[Proof of Theorem~\ref{thmintdist}]
Suppose $\Mod \Omega \le \Mod \Omega^\ast$ and let $f_\circ \colon \Omega \onto \Omega^\ast$ be an energy-minimal diffeomorphism provided to us by Theorem~\ref{q4}. For every homeomorphism $g \colon \Omega^\ast \onto \Omega$ with integrable distortion the inverse map $f=g^{-1} \colon \Omega \onto \Omega^\ast$ belongs to the Sobolev class $W^{1,2}(\Omega)$ and we have~\cite{HK, HKO, IS}
\begin{equation}\label{blah}
\int_{\Omega^\ast} K_g(w)\, d w = \int_\Omega \abs{Df(z)}^2 \, dz \ge \int_\Omega \abs{Df_\circ (z)}^2\, dz=  \int_{\Omega^\ast} K_{g_\circ} (w)\, dw
\end{equation}
where $g_\circ=f_\circ^{-1}$. The latter identity is legitimate because $f_\circ$ is a diffeomorphism. Thus $g_\circ$ is also a $C^\infty$-smooth diffeomorphism. It has the smallest possible $L^1$-norm of the distortion. If equality holds in~\eqref{blah} then, by Theorem~\ref{q4}, the mapping $f_\circ^{-1} \circ f$ is conformal.
\end{proof}

\section{Nonexistence: Theorem~\ref{emexist}}\label{nonsec}

Before proceeding to the proof of  Theorem~\ref{emexist} we recollect basic facts of potential theory in the plane which can be found in~\cite{Ranb}.
A domain $\Omega$ has Green's function $\G_{\Omega}$ whenever $\C\setminus \Omega$ contains
a nondegenerate continuum. Our normalization is $\G_{\Omega}(z,\zeta)=-\log\abs{z-\zeta}+O(1)$ as $z\to\zeta$.
In particular, $\G_{\Omega}(z,\zeta)>0$. Green's function for the unit disk $\DD$ is
\[ 
\G_{\DD}(z,\zeta)=\log \left|\frac{1-z \bar \zeta}{z-\zeta}\right|.
\] 
If $f\colon \Omega\to\Omega^*$ is a holomorphic function, then the subordination principle holds:
\begin{equation}\label{gsub}
\G_{\Omega}(z,\zeta)\le \G_{\Omega^*}(f(z),f(\zeta)).
\end{equation}

\begin{proof}[Proof of Theorem~\ref{emexist}]
If $\Omega$ is degenerate, so is $\Omega^\ast$ because a point is  a removable singularity for $W^{1,2}$-homeomorphisms~\cite[Theorem 3.1]{IOt}.
Therefore we may assume, by a conformal change of variables in $\Omega$, that $\Omega=A(R^{-1},R)$, $R>1$. By Lemma~\ref{ctheory}
\begin{equation}\label{repeat}
h_z\overline{h_{\bar z}} \equiv \frac{c}{z^2}\qquad \text{in }\Omega.
\end{equation}
where $c$ is real. If $c\ge 0$, then~\eqref{important} yields $\abs{h_T}^2 \le J_h$, hence $\Mod \Omega^\ast \ge \Mod \Omega$ by Proposition~\ref{rwthe}.
It remains to consider the case $c<0$. Let us write $c=-b^2$, $b\in\R$.
Introduce the so-called second complex dilatation
\begin{equation}\label{rw2}
\nu =\frac{\overline{h_{\bar z}}}{h_z}
\end{equation}
which is a holomorphic function from $\Omega$ into the unit disk $\DD$~\cite[p. 5]{Dub}.
Equation~\eqref{repeat} implies that $\nu$ does not vanish and
\[
\nu = \frac{-b^2}{z^2 h_z^2}
\]
Therefore, $\nu$ has a single-valued square root, namely
\begin{equation}\label{rw4}
\omega=\frac{ib}{z h_z}.
\end{equation}
From~\eqref{rw2} and~\eqref{rw4} we have
\begin{equation}\label{rw5}
h_z = \frac{ib}{z\omega}\quad \text{ and } \quad h_{\bar z}=-\frac{ib\overline{\omega}}{\bar z}.
\end{equation}
Now we integrate the differential form $dh=h_z\,dz+h_{\bar z}\, d\bar z$ over the unit circle
\begin{equation}\label{rw6}
0 = \int_{\T}dh =ib \int_{\T}\left(\frac{dz}{ z\omega}-\frac{\overline{\omega }\,d\bar z}{\bar z}\right)
=ib \int_{\T}\left(\frac{1}{ \omega}+\overline{\omega }\right)\frac{dz}{z}.
\end{equation}
The image of $\T$ under the map $z\mapsto\frac{\omega}{\abs{\omega}}$ is an arc $\Gamma\subset \T$.
This arc cannot be contained in any open half-circle, for
then the values of the function $\omega^{-1}+\overline{\omega}=\overline{\omega}(1+\abs{\omega}^{-1})$ on $\T$ would lie in an open halfplane,
contradicting~\eqref{rw6}.
Thus there exist points $z_1,z_2\in \T$ such that
\begin{equation}\label{opp}
\frac{\omega(z_1)}{\omega(z_2)}<0.
\end{equation}
We write $w_j=\omega(z_j)$, $j=1,2$, and invoke a simple lower bound for the Green function of $\Omega$, derived in~\cite[(3.9)]{IKO4}:
\begin{equation}\label{gl}
\G_{\Omega}(z_1,z_2)\ge \log \coth\frac{\pi^2}{4\log R},\qquad z_1,z_2\in\T.
\end{equation}
By the subordination principle~\eqref{gsub},
\begin{equation}\label{sub}
\G_{\Omega}(z_1,z_2) \le \G_{\DD}(w_1,w_2).
\end{equation}
Because of symmetry we may assume $\abs{w_1}\le \abs{w_2}$.
The right hand side of~\eqref{sub} is estimated from above using~\eqref{opp}:
\begin{equation}\label{gb}
\G_{\DD}(w_1,w_2) = \log\frac{1+\abs{w_1w_2}}{\abs{w_1}+\abs{w_2}}
\le\log\frac{1+\abs{w_1}^2}{2\abs{w_1}}.
\end{equation}
Combining~\eqref{gl}--\eqref{gb} we obtain an upper bound for $\abs{\omega}$ on $\T$,
\begin{equation}\label{rw7}
\abs{\omega(z_1)} \le \tanh\frac{\pi^2}{8\log R}.
\end{equation}
Introduce an auxiliary mapping $g=\phi\circ h$, where $\phi$ is an affine transformation chosen so that $g$ becomes conformal at $z_1$; that is, $g_{\bar z} (z_1)=0$.
It was proved in~(\cite{IKO4}, estimates  (3.11) and (3.13)) that
\begin{equation*}
\Mod g(\Omega) \ge \Mod\Omega \cdot  \Lambda\left(\coth \frac{\pi^2}{2\tau}\right),\quad \ \Lambda(t)= \frac{\log t-\log(1+\log t)}{2+\log t}, \ t\ge 1.
\end{equation*}
From~\eqref{rw7} we have
\begin{equation*}
\Mod h(\Omega) \ge \frac{1-\abs{\omega(z_1)}}{1+\abs{\omega(z_1)}}\Mod g(\Omega)
=\exp\left(-\frac{\pi^2}{4\log R}\right)\Mod g(\Omega).
\end{equation*}
Combining the last two lines yields~\eqref{specific}.
\end{proof}

We complement Theorem~\ref{emexist} with an explicit example of two doubly connected domains  and a harmonic homeomorphism between them, which do not admit an energy-minimal homeomorphism.

\begin{example}\label{ellip} Consider the annulus $\Omega=A(1,R)$, $R>1$. Fix $0<\delta<1$ and let $\Omega^*$
be the image of the annulus $\A^\ast=A(1, \frac{1}{2} (R+R^{-1}))$
under the affine mapping $\phi(z) = z + \delta \bar z$.
Then there exists no energy-minimal diffeomorphism $h\colon \Omega\onto\Omega^*$, though
there exists a harmonic one.
\end{example}

\begin{proof} The annulus $\A^\ast$ is the image of $\Omega$ under the
\emph{extremal Nitsche mapping}
\[
h^*(z)=\frac{1}{2}\left(z+\frac{1}{\bar z}\right)
\]
which is not only harmonic but also energy-minimal in
$\Ho^{1,2}(\Omega,\A^\ast)$~\cite[Corollary 2]{AIM}. The uniqueness part of Theorem~1.1 in~\cite{IKO2} states that $h^*$ is the unique harmonic homeomorphism from $\Omega$ onto $\A^\ast$, up to a conformal automorphism of the annulus $\Omega$, rotation or/and inversion. It follows that $g:=\phi\circ h^*$ is the unique harmonic diffeomorphism of $\Omega$ onto $\Omega^*$, up
to a conformal automorphism of $\Omega$.
Thus if $\Ho^{1,2}(\Omega,\Omega^\ast)$ admitted an energy minimizer the mapping $g$ would be one of them.
Explicitly,
\[
g(z) = \frac{1}{2}\left(z+\frac{\delta}{z}+\delta \bar z+\frac{1}{\bar z} \right).
\]
On the other hand, the Hopf differential of $g$ takes the form
\[
g_z \overline{g_{\bar z}} =\frac{1}{4}\left(1-\frac{\delta}{z^2}\right)\left(\delta-\frac{1}{z^2}\right) \not\equiv \frac{c}{z^2}.
\]
By Lemma~\ref{ctheory} we see that $g$ cannot be stationary in the annulus $\Omega$. Consequently, there is no  energy-minimal homeomorphism in $\Ho^{1,2}(\Omega,\Omega^*)$.
\end{proof}


\section{Convexity of the minimum energy function}\label{consec}

In \S\ref{monsec} we proved that for any bounded doubly connected domain $\Omega^*$
the function $\EE(\tau,\Omega^*)$ is decreasing for $0<\tau<\Mod\Omega^*$. The
minimum of this function is attained at  $\tau=\Mod\Omega^*$, i.e., in the case of conformal equivalence.
In this section we prove:

\begin{theorem}\label{convextheorem}
Let $\Omega^*$ be a bounded doubly connected domain.
The function $\tau\mapsto \EE(\tau,\Omega^*)$ is strictly convex for $0<\tau<\Mod\Omega^*$.
\end{theorem}

The main part of the proof of this theorem needs to be  stated separately.
As a by-product it establishes the uniqueness part of Theorem~\ref{mainexist}.

\begin{proposition}\label{needed}
Let $\Omega^\ast$ be a bounded doubly connected domain.
Suppose that
$h\in \D(A(\tau_\circ),\Omega^\ast)$ is an energy-minimal deformation. In particular, by Lemma~\ref{ctheory},
\begin{equation}\label{hopf1again}
h_z\overline{h_{\bar z}} \equiv \frac{c}{z^2}\qquad \text{in }A(\tau_\circ).
\end{equation}
Then for any diffeomorphism $g\colon A(\tau) \to \Omega^\ast$ we have
\begin{equation}\label{claim10}
\E[g]-\E[h] \ge  8\pi c (\tau_\circ-\tau).
\end{equation}
If, in addition, $h$ is a diffeomorphism, then equality holds in~\eqref{claim10} if and only if $\tau=\tau_\circ$ and $g^{-1} \circ h$ is
a conformal mapping of $A(\tau_\circ)$ onto itself.
\end{proposition}

\begin{proof} First we dispose of the easy case $c=0$. In this case $h$ is conformal, which implies
$\E[h]=2\abs{\Omega^*}$. On the other hand, $\E[g]\ge 2\abs{\Omega^*}$ with equality if and only if $g$ is conformal, see~\eqref{ener2}.

It remains to deal with $c\ne 0$. The composition
\[
f= g^{-1} \circ h \colon A(\tau_\circ) \onto A(\tau)
\]
lies in $W^{1,2}_{\rm loc} (A(\tau_\circ))$ and is not homotopic to a constant mapping.  Moreover, the restriction of $f$
to the domain $G:=\{z\in A(\tau_\circ)\colon h(z)\in \Omega^*\}$ is a harmonic diffeomorphism onto $A(\tau)$,
by virtue of Proposition~\ref{gtheory}. Thus, $f$ possesses a right inverse $f^{-1}\colon A(\tau)\onto G$
which is also a diffeomorphism. We estimate  $\E[g]-\E[h]$ in several steps.
The first step is to apply the chain rule to the derivatives of $g=h\circ f^{-1}(w)$ at $w=f(z)$.
\begin{equation}\label{gwder}
\begin{split}
\frac{\partial g}{\partial w} & =\frac{h_z\overline{f_z}-h_{\bar z}\overline{f_{\bar z}}}{J_f} \\
\frac{\partial g}{\partial \bar w} & =\frac{h_{\bar z}{f_z}-h_{\bar z}{f_{\bar z}}}{J_f}
\end{split}
\end{equation}
Then by change of variables the Dirichlet energy of $g$ in $A(\tau)$ reduces to an integral over $G$.
\begin{equation*}
\begin{split}
\E[g]&=2\int_{A(\tau)} \left(\abs{g_w}^2+\abs{g_{\bar w}}^2\right)\,d w \\
&=2\int_{G}\frac{\abs{h_z \overline{f_z} - h_{\bar z} \overline{f_{\bar z}}}^2+
\abs{h_{\bar z} f_z - h_{z} f_{z}}^2}{J_f}\, d z
\end{split}
\end{equation*}
Next, subtract $\int_{G} \abs{Dh}^2$ from $\E[g]$, use the inequality $\abs{h_z}^2+\abs{h_{\bar z}}^2\ge 2 \abs{h_z h_{\bar z}}$, and recall~\eqref{hopf1again} to obtain
\begin{align}
\E[g]-\int_{G} \abs{Dh}^2 &= 4\int_{G}
\frac{\left(\abs{h_z}^2+\abs{h_{\bar z}}^2\right)\, \abs{f_{\bar z}}^2
-2\re \left[h_z \overline{h_{\bar z}} \overline{f_z}f_{\bar z}\right]}{J_f}\, d z \nonumber \\
& \label{cchain} \ge 4 \int_{G} \frac{ 2\abs{h_z h_{\bar z}}\, \abs{f_{\bar z}}^2
-2\re \left[h_z\overline{h_{\bar z}} \overline{f_z}f_{\bar z}\right]}{J_f}\, d z \\
& = 4\abs{c} \int_{G} \left[\frac{\abs{f_z-\sigma f_{\bar z}}^2}{J_f} -1 \right]
\, \frac{dz}{\abs{z}^2}, \quad \mbox{where } \sigma=\sigma(z)= \frac{c\bar z}{\abs{c}z} \nonumber
\end{align}
We must also account for the integral of $\abs{Dh}^2$ over $A(\tau_\circ)\setminus G$. On this set $J_h=0$ a.e. by
Lemma~\ref{goodset}, which in view of~\eqref{hopf1again} implies
\[
\abs{h_z}^2+\abs{h_{\bar z}}^2 = 2\abs{h_z}^2 = \frac{2\abs{c}}{\abs{z}^2}.
\]
Hence
\begin{equation}\label{badeasy}
\int_{A(\tau_\circ)\setminus G} \abs{Dh}^2 = 4\abs{c} \int_{A(\tau_\circ)\setminus G} \frac{dz}{\abs{z}^2}.
\end{equation}
Combining~\eqref{cchain} and~\eqref{badeasy} we arrive at
\begin{equation}\label{newcchain}
\begin{split}
\E[g]-\E[h] &\ge  4\abs{c} \int_{G} \frac{\abs{f_z-\sigma f_{\bar z}}^2}{J_f}\,\frac{dz}{\abs{z}^2}
\, - 4\abs{c}\int_{A(\tau_\circ)} \frac{dz}{\abs{z}^2} \\ &=
4\abs{c} \int_{G} \frac{\abs{f_z-\sigma  f_{\bar z}}^2}{J_f}\, \frac{dz}{\abs{z}^2}
-\, 8\pi \abs{c} \tau_\circ.
\end{split}
\end{equation}

At this stage the sign of $c$ comes into play. Note that
\[
\frac{\abs{f_z-\sigma f_{\bar z}}^2}{J_f}
= \begin{cases} K_T^f \quad \text{if } c>0 \\ K_N^f \quad \text{if } c<0.
\end{cases}
\]
Lemma~\ref{goodset} tells us that the Jacobian $J_h$ vanishes almost everywhere on $A(\tau_\circ)\setminus G$. This  together
with~\eqref{important} imply that one of directional derivatives of $h$ must vanish a.e. on $A(\tau_\circ)\setminus G$:
$h_T=0$ if $c>0$ or $h_N=0$ if $c<0$. Since $f=g^{-1}\circ h$, the same alternative applies to the directional derivatives  $f_T$ and $f_N$.
In summary, the last integral in~\eqref{newcchain} may as well be taken over $A(\tau_\circ)$ instead of $G$.
\begin{equation}\label{newcc}
\E[g]-\E[h] \ge \begin{cases}
\displaystyle 4\abs{c} \int_{A(\tau_\circ)} K_T^f\,\frac{dz}{\abs{z}^2}  -\, 8\pi \abs{c} \tau_\circ  \quad \text{if } c>0
\vspace{0.3cm}
\\
\displaystyle 4\abs{c} \int_{A(\tau_\circ)} K_N^f\,\frac{dz}{\abs{z}^2}  -\, 8\pi \abs{c} \tau_\circ \quad \text{if } c<0.
\end{cases}
\end{equation}

In the case $c>0$ we apply Proposition~\ref{rwthe} to $f$ and obtain the estimate
\begin{equation}\label{RSWthe7}
\int_{A(\tau_\circ)} K_T^f \, \frac{d z}{\abs{z}^2} \, \ge \, 2\pi \frac{\tau_\circ^2}{\tau}
\end{equation}
which together with~\eqref{newcc} yield
\begin{equation}
\E[g]-\E[h] \ge 8\pi c \frac{\tau_\circ}{\tau}(\tau_\circ-\tau) \ge 8\pi c (\tau_\circ-\tau), \qquad 0<\tau_\circ,\tau<\infty.
\end{equation}

If $c<0$, then Proposition~\ref{rwrho}~(b) applies to the restriction of $f$ to $G$, yielding
\begin{equation}\label{RSWrho7}
\int_{G} K_N^f \, \frac{d z}{\abs{z}^2} \ge 2\pi {\tau}
\end{equation}
which together with~\eqref{newcchain} imply~\eqref{claim10}.

It remains to prove the equality statement.
Since $h$ is a sense-preserving diffeomorphism, we have $G=A(\tau_\circ)$ and $\abs{h_z}>\abs{h_{\bar z}}$ everywhere in $A(\tau_\circ)$.
If equality holds in~\eqref{claim10}, then it also holds in~\eqref{cchain}. The latter is only possible if
$f_{\bar z}\equiv 0$ in $A(\tau_\circ)$. Thus $f\colon A(\tau_\circ)\onto A(\tau)$ is a conformal mapping. This implies $\tau_\circ=\tau$, as desired.
\end{proof}

\begin{proof}[Proof of Theorem~\ref{convextheorem}]
Pick $\tau_\circ \in (0, \Mod \Omega^\ast)$. By Theorem~\ref{emexist} there exists $h\in \Ho^{1,2} (A(\tau_\circ), \Omega^\ast)$ such that $\E[h]=\EE(\tau_\circ,\Omega^*)$.  Consequently,~\eqref{hopf1again} holds. That $c>0$ follows from Corollary~\ref{cpositive}. We now claim that
\begin{equation}\label{congoal}
\EE(\tau,\Omega^*)- \EE(\tau_\circ,\Omega^*) >  -8 \pi  c (\tau-\tau_\circ), \qquad \tau \in
(0, \Mod \Omega^\ast) , \quad \tau \ne \tau_\circ
\end{equation}
Indeed, by Theorem~\ref{emexist} there exists $g\in \Ho^{1,2} (A(\tau), \Omega^\ast)$ such that $\E[g]=\EE(\tau,\Omega^*)$. Proposition~\ref{needed} is exactly what we need for~\eqref{congoal}.

Inequality~\eqref{congoal} tells us that $\EE(\tau,\Omega^*)$ is   strictly convex.
Together with~\eqref{econt1} it yields the existence of the derivative
\[
\frac{d}{d\tau}\bigg|_{\tau=\tau_\circ}\EE(\tau,\Omega^*) = -8\pi c,
\]
Incidentally or not, this shows that $c$ depends only on $\tau_\circ$ and $\Omega^*$, but not on $h$.
Every convex function, once differentiable everywhere, is automatically $\CC^1$-smooth; the theorem
is fully established.
\end{proof}

The strict convexity part of Theorem~\ref{convextheorem} fails for $\tau>\Mod\Omega^*$. We demonstrate this with an example based on the results of~\cite{AIM}. Although the paper~\cite{AIM}
is concerned with the minimization of energy in a somewhat different class of Sobolev mappings, its approach carries over to our setting with no changes.

\begin{example}\label{circ} Let $\Omega^*=A(1,R_*)$ where $1<R_*<\infty$.
The function $\tau\mapsto \EE(\tau,\Omega^*)$ is $\CC^2$-smooth on $(0,\infty)$,
strictly convex for $0<\tau<\log \cosh\Mod\Omega^*$ and affine for $\tau>\log \cosh \Mod\Omega^*$.
\end{example}

\begin{proof}
Let $\Omega=A(1,R)$ where $R=e^\tau$.
We begin with the case $0<\tau<\log \cosh\Mod\Omega^*$.  In terms of $R$ this condition reads as
\begin{equation}\label{circ0}
R_*\ge \frac{1}{2}\left(R+\frac{1}{R}\right), \qquad \text{equivalently,  }\; \; R\le R_*+\sqrt{R_*^2-1}.
\end{equation}
Let $\lambda\in (-1,1]$ be determined by the equation
\begin{equation}\label{circ1}
R_*=\frac{1}{1+\lambda}\left(R+\frac{\lambda}{R}\right);\quad  \mbox{ that is, }\; \;  \lambda=\frac{R(R-R_*)}{RR_*-1}.
\end{equation}
By~\cite[Corollary~2]{AIM} the infimum of energy $\EE(\Omega,\Omega^*)$ is achieved by the mapping
\begin{equation}
h^{\lambda}(z)=\frac{1}{1+\lambda}\left(z+\frac{\lambda}{\bar z}\right),
\end{equation}
for which we compute
\begin{equation}\label{circ2}
\E[h^{\lambda}] = 2\pi\frac{(R^2-1)(R^2+\lambda^2)}{R^2(1+\lambda)^2}.
\end{equation}
which  yields
\begin{equation}\label{circ3}
\EE(\log R,\Omega^*) = 2\pi \frac{(R^2+1)[(R_*^2+1)-4RR_*]}{R^2-1}.
\end{equation}
A straightforward computation reveals that the righthand side of~\eqref{circ3} is a convex function of $\log R$ in the range given by~\eqref{circ0}. Indeed, its derivative with respect to $\log R$ is equal to
\begin{equation}\label{circ5}
R\frac{d }{dR}\EE(\log R,\Omega^*) = \frac{8\pi R (R-R_*)(RR_*-1)}{(R^2-1)^2}.
\end{equation}
Differentiating~\eqref{circ5} once again, we find
\begin{equation}\label{circ6}
\begin{split}
R\frac{d }{dR} &\left(R\frac{d }{dR}\EE(\log R,\Omega^*)  \right)  \\
&= \frac{8\pi R}{(R^2-1)^3} \left\{(R_*R^2-2R+R_*)(2RR_*-R^2-1)\right\}.
\end{split}
\end{equation}
The right hand side of~\eqref{circ6} has the same sign as $(2RR_*-R^2-1)$, which proves the claim.
For future reference we note that at the transition point $R= R_*+\sqrt{R_*^2-1}$
the equations~\eqref{circ5} and~\eqref{circ6} yield one-sided derivatives of $\EE(\tau,\Omega^*)$, namely
\begin{equation}\label{transi}
\frac{d}{d\tau} \EE(\tau,\Omega^*) = 2\pi, \qquad
\frac{d^2}{d\tau^2} \EE(\tau,\Omega^*) =0.
\end{equation}

It remains to consider the case $\tau>\log \cosh\Mod\Omega^*$.
Now it is more convenient to work with   $\Omega=A(r,R)$ where $R=R_*+\sqrt{R_*^2-1}$ and $r<1$.
By~\cite[Theorem 1.8]{IO} the infimum $\EE(\Omega,\Omega^*)$ is realized by a non-injective deformation $h \colon \Omega \onto \Omega^\ast$.
\[h=\begin{cases} \frac{z}{\abs{z}} & \mbox{for } r<\abs{z} \le 1\\
\frac{1}{2} \left(z+\frac{1}{\bar z}\right) & \mbox{ for } 1\le \abs{z} <R \end{cases}\]
Here the radial projection $z\mapsto z/\abs{z}$
hammers $A(r,1)$ onto the unit circle while the Nitsche mapping $\frac{1}{2} \left(z+ \frac{1}{\bar z}\right)$ takes $A(1,R)$ homeomorphically onto $\Omega^\ast$.
The contribution of the radial projection to the energy of $h$ is equal to
\begin{equation}
2\pi \log \frac{1}{r} = 2\pi (\tau-\log \cosh\Mod\Omega^*).
\end{equation}
This is an affine function of $\log R$ whose first derivative equals  $2\pi$ and the second derivative vanishes. This result remains in agreement with formulas~\eqref{transi}. Thus  $\EE(\tau,\Omega^*)$ is a $\CC^2$-smooth function.
\end{proof}

\section{Open questions and conjectures}\label{quesec}

In~\eqref{en2} and~\eqref{en1} we defined   two infima of energy; the one denoted $\EE_\Ho(\Omega,\Omega^*)$ runs over homeomorphisms and the other, $\EE(\Omega,\Omega^*)$, over deformations in the sense of
Definition~\ref{defdef}). Clearly $\EE_\Ho(\Omega,\Omega^*)\ge \EE(\Omega,\Omega^*)$. Under the hypotheses of Theorem~\ref{q4}  $\EE_\Ho(\Omega,\Omega^*)= \EE(\Omega,\Omega^*)$.

\begin{question}\label{obvious}
For $k\ge 2$, is $\EE_\Ho(\Omega,\Omega^*) = \EE(\Omega,\Omega^*)$  for all $k$-connected bounded domains $\Omega$ and $\Omega^\ast$ in $\C$?
\end{question}

This question  has the affirmative answer in the case $k=1$ thanks to the Riemann mapping theorem.   Indeed, due to Corollary~\ref{useless} the formula~\eqref{ener2} remains valid for all deformations. Therefore, the conformal mapping minimizes the energy.

Theorem~\ref{convextheorem} and Example~\ref{circ} motivate the following conjecture.

\begin{conjecture}
The function $\tau\mapsto \EE(\tau,\Omega^*)$ is convex for $0<\tau<\infty$.
\end{conjecture}

Note that it would follow from the positive answer to Question~\ref{obvious}, by means of Proposition~\ref{needed}.

We expect that Theorem~\ref{emexist} can be given the following sharp form.

\begin{conjecture}
If two bounded doubly connected domains  $\Omega$ and $\Omega^*$ in $\C$ admit an energy-minimal
diffeomorphism $h\colon \Omega\onto\Omega^*$, then
\[
\Mod\Omega^\ast\ge \log \cosh \Mod \Omega.
\]
Moreover, if both sides are finite and equal, then $\Omega^*$ is a circular annulus.
\end{conjecture}

Concerning the existence of energy-minimal diffeomorphisms between
domains of higher connectivity, we propose a
generalization of Theorem~\ref{mainexist}.

\begin{conjecture}\label{kexist} Let $\Omega$ and $\Omega^*$ be
bounded $k$-connected domains in $\C$, where $k\ge 2$.
Suppose that $\Omega\subset \Omega^*$ where the inclusion  is a
homotopy equivalence. Then there exists
an energy-minimal diffeomorphism of $\Omega$ onto $\Omega^*$.
\end{conjecture}

For $k=2$ Conjecture~\ref{kexist} is true, by virtue of
Theorem~\ref{mainexist}. In the converse direction, we propose
a qualitative version of Theorem~\ref{emexist} for $k$-connected domains.

\begin{conjecture} Let $\Omega$ and $\Omega^*$ be bounded
$k$-connected domains in $\C$, where $k\ge 2$.
If $\epsilon>0$ is sufficiently small (depending on both $\Omega$ and $\Omega^\ast$), then there is no energy-minimal
homeomorphism of $\Omega$ onto $\phi(\Omega^*)$, where $\phi(x+iy)=\epsilon x+iy$.
\end{conjecture}
 In other words, if we flatten $\Omega^\ast$ too much in one direction the injectivity of  energy-minimal deformations $f \in \D(\Omega, \Omega^\ast)$ will be lost.

\section{Appendix: Monotone Sobolev mappings}\label{appen}

Throughout this section $\X$ will be a bounded domain in $\C$ whose complement consists of $k$ mutually disjoint closed connected sets denoted by
\[\C \setminus \X =  \x_1 \cup \dots \cup \x_k =:\x , \qquad k\ge 2.\]
It then follows that to every $\x_i$ there corresponds one and only one component of $\partial \X$, precisely equal to $\partial \x_i$,
\[\partial \X = \partial \x= \partial \x_1 \cup \dots \cup \partial \x_k .\]
Among those components there is exactly one unbounded. Similarly to $\X$, we consider a bounded domain in $\Y \subset \C$ whose complement consists of $k$-mutually disjoint closed connected sets denoted by
\[\C \setminus \Y = \y_1 \cup \dots \cup \y_k =:\y.\]
We make one standing assumption on $\X$; namely, none of the components $\x_1, \dots , \x_k$ degenerates to a single point.
\begin{equation}\label{smallestdiam}
\min_{1\le i \le k} \diam \x_i =d >0
\end{equation}
Similar assumption on $\Y$ will not be required. Let us denote
\begin{equation}
 \rho_{{}_{\Y}}= \inf_{\alpha \ne \beta} \dist (\y_\alpha, \y_\beta)>0.
\end{equation}
We shall examine the class $\mathcal F_{\Y} (\X)$ of mappings  $h \colon \X \to \C$ such that 
\begin{enumerate}[(i)]
\item $h\in C(\X) \cap W^{1,2}(\X)$;
\item  $h(\X) \supset \Y$;
\item\label{clust} $h\{\partial \x_i\} \subset \partial \y_i$,  $i=1, \dots , k$, in the sense of cluster sets;
\item\label{monoapp}  the restriction of $h$ to $h^{-1}(\Y)$ is monotone.
\end{enumerate}

It follows from~\eqref{clust}  that $h^{-1}(\Gamma)$ is compact  for any compact set $\Gamma \subset \Y$. If in addition $\Gamma$ is  connected, then $h^{-1}(\Gamma)$ is connected by Proposition~\ref{why}.
\begin{lemma}\label{lemapp1}
There is a constant $c=c(\X, \Y)>0$ such that
\begin{equation}
\E[h] \ge c(\X, \Y), \qquad \mbox{ for every } h \in \mathcal F_\Y (\X).
\end{equation}
In fact, we have the following explicit bound.
\[\int_\X \abs{Dh}^2 \ge \frac{ \rho^2_{{}_{\Y}} \, d }{\diam \X}. \]
\end{lemma}

\begin{proof}
Choose  a bounded component $\x_i$. Let a line segment $I$ with the end-points in $\x_i$ represent the diameter of $\x_i$; thus $\abs{I}= \diam \x_i$.  Through every point $t\in I$ there passes a straight line $L_t$ perpendicular to $I$. One of the components of $\X \cap L_t$, say an open interval $\gamma$, connects $\partial \x_i$ with $\partial \x_\alpha$, for some $\alpha \ne i$. Thus, by condition~\eqref{clust},
\[\int_{\X \cap L_t} \abs{Dh} \ge \int_\gamma \abs{Dh} \ge \dist (\y_\alpha, \y_i) \ge  \rho_{{}_{\Y}} .\]
This is true for almost every $t\in I$, as long as $h$ is locally absolutely continuous on $\X \cap L_t$. By H\"older's inequality
\[\int_{\X\cap L_t} \abs{Dh}^2 \ge \frac{1}{\abs{\X \cap L_t}} \left(\int_{\X \cap L_t} \abs{Dh}  \right)^2 \ge \frac{ \rho^2_{{}_{\Y}}}{\diam \X}.\]
Integrating with respect to $t\in I$, by Fubini's theorem, we conclude that
\[ \int_\X \abs{Dh}^2 \ge \int_I \left( \int_{\X \cap L_t} \abs{Dh}^2  \right)\, dt \ge \frac{\diam \x_i}{\diam \X}\,  \rho^2_{{}_{\Y}} \ge  \frac{ \rho^2_{{}_{\Y}} \, d }{\diam \X} \]
as desired.
\end{proof}

\begin{theorem}\label{limittheorem}
For each $i=1, \dots , k$ there exists a continuous function $\eta_i=\eta^i_{{}_{\X , \Y}} (z)$ on $\C$, vanishing on $\x_i$, such that for every $h\in \mathcal F_\Y (\X)$ we have
\begin{equation}\label{distineq}
\dist \big(h(z), \y_i\big) \le \eta_i (z) \sqrt{\E[h]}, \qquad z\in\X.
\end{equation}
\end{theorem}

\begin{proof}
It suffices to construct for each $i=1, \dots , k$, a function $\eta_i = \eta_i(z)$ in $\X$ which is bounded and satisfies the conditions $\lim\limits_{z\to \x_i} \eta_i (z)=0$ and~\eqref{distineq}  for all $h\in \mathcal F_{\Y} (\X)$. Continuity of $\eta_i$ can easily be accomplished by taking a continuous majorant. The obvious choice for $\eta_i$ is:
\begin{equation}
\eta_i (z)= \sup_{h\in \mathcal F (\X, \Y)} \frac{\dist \big(h(z), \y_i  \big)}{\sqrt{\E[h]}}, \qquad i=0,1, \dots , k.
\end{equation}
By Lemma~\ref{lemapp1} we see that $\eta_i(z) \le \frac{\diam \Y}{\sqrt{c(\X, \Y)}}$. Fix an index $i$ and suppose, to the contrary, that
$\lim\limits_{z \to \x_i} \eta_i (z) \ne 0$.
Then $\eta_i (z_\nu) \ge \epsilon >0$ for some sequence $\{z_\nu\} \subset \X$ converging to a point $z_\circ \in \x_i$. This means that there is a sequence $\{h_\nu\}$ of functions in $\mathcal F_\Y (\X)$ such that
\begin{equation}\label{badnu}
\dist \big( h_\nu (z_\nu), \y_i \big) \ge \epsilon \, \sqrt{\E [h_\nu]} \ge \epsilon \, \sqrt{c (\X, \Y)},
\end{equation}
by Lemma~\ref{lemapp1}. Obviously, we have
\begin{equation}\label{obviously}
\E[h_\nu] \le \left( \frac{\diam \Y}{\epsilon} \right)^2 , \qquad \nu =1,2, \dots
\end{equation}
Choose and fix  a doubly connected domain $G\subset \Y$ so that one of the connected components of $\C\setminus G$ is $\y_i$.
The following lemma provides us with what we call a potential function for $\y_i$.
\begin{claim0}
There exists a $C^1$-smooth function $U \colon \C \to [0,1]$ such that
$U^{-1} \{0\}=\y_i$ and $U^{-1} \{1\}$ is precisely the other connected component of $\C\setminus G$. Moreover,  for each  $0< t < 1$ the set $\Gamma_t= U^{-1} \{t\}$ is a Jordan curve separating the boundary components of $G$.
\end{claim0}

\begin{proof} Let $\Phi\colon G\to A(r,R)$ be a conformal mapping of $G$ onto a circular annulus $A(r,R)$ or a punctured disk. The function $\abs{\Phi}$ has the desired structure of level sets but may lack smoothness on the boundary. The latter is remedied with the help of a smooth
strictly increasing function $\psi\colon (r,R)\to (0,1)$ such that $\psi'\to 0$ sufficiently fast at the points $r$ and $R$. We define $U$ as the composition $\psi(\abs{\Phi})$, extended by $0$ and $1$ to the entire plane $\C$.
\end{proof}

For $h\in \mathcal F_\Y (\X)$ we consider the continuous function
\begin{equation}\label{defofv}
V(z)=V_h(z)=\begin{cases} U(h(z))\qquad & \text{if }\ z\in\X \\
0 & \text{if }\ z\in\x_i \\
1 & \text{if }\ z\notin \X\cup\x_i.
\end{cases}
\end{equation}
The continuity of $V$ follows from the condition~\eqref{clust}  after
taking into account that $U(\y_i)=\{0\}$ while $U(\y_\alpha)=\{1\}$ for $\alpha\ne i$.

Recall the constant $d$ that was defined in~\eqref{smallestdiam} as the smallest of the numbers $\diam \x_i$, $i=1,\dots,k$.

\begin{claima}
For any $h\in \mathcal F_\Y (\X)$ and $0<t<1$ the level set $V_h^{-1}\{t\}$ is a continuum of diameter at least $d$.
\end{claima}
\begin{proof}
That $V_h^{-1}\{t\}$ is a continuum follows from the monotonicity assumption~\eqref{monoapp}. Choose $\alpha$ such that: $\alpha=i$ if $\x_i$ is bounded and $\alpha\ne i$ otherwise.
In either case the component $\x_\alpha$ is bounded.
Consider a straight line $L$ passing through two points $a,b \in \partial \x_\alpha$ such that $\abs{a-b}= \diam \x_\alpha$.
The set $L\setminus (a,b)$ consists of two closed half-lines $L_a$ and $L_b$. We will show that on each of them $V_h$ attains
the value $t$, which yields 
\[\diam V_h^{-1}\{t\} \ge \abs{a-b} = \diam \x_\alpha \ge d.\]

The half-line $L_a$ meets a bounded component $\x_\alpha$ at the point $a$, and must also intersect the unbounded component of $\C\setminus \X$.
Considering our choice of $\alpha$ we find that $L_a$ meets both $\x_i$ and some other component of $\C\setminus \X$.
Thus $V_h$ attains the values $0$ and $1$ on $L_a$. Being continuous, it also attains the value $t$. Similarly we argue with the half-line $L_b$.
\end{proof}

\begin{claimb} We have $V\in W_{\rm loc}^{1,2}(\C)$. Moreover, $V$ has the oscillation property on every open disk $\B\subset \C$
of diameter not greater than $d=\min\limits_{1\le \alpha\le k}\diam \x_\alpha$.
\end{claimb}
\begin{proof} First note that
$V\in W^{1,2}(\X)$ and we have the pointwise estimate
\[
\abs{\nabla V(z)}\le \norm{\nabla U}_{L^{\infty}(\C)} \abs{Dh(z)},\qquad z\in \X.
\]
Recall that  $V$ is continuous on $\C$ and is constant on each component of $\C\setminus \X$.
The classical Sobolev theory tells us that such function belongs to  $W^{1,2}_{\rm loc}(\C)$
with the energy bound
\begin{equation}\label{tristar}
\int_C \abs{\nabla V(z)}^2\,dz \le \norm{\nabla U}_{\infty} \int_{\X}\abs{Dh}^2.
\end{equation}

We now proceed to check the oscillation property of $V$ on a disk $\B\subset\C$. For this we choose a compact set $\mathbb F\subset \B$.
Consider an arbitrary component of $\mathbb F$, denoted  $\mathbb F_\circ$. Note that $\diam \mathbb F_{\circ}<\diam\B\le d$.
The set $V(\mathbb F_\circ)$ is a compact subinterval of $[0,1]$ which we denote by $[a,b]$.
We will show that 
\begin{equation}\label{wws}
[a,b]\subset V(\partial \mathbb F_\circ).
\end{equation}  

This is obvious when $a=b$, for then $V$ is constant on $\mathbb F_\circ$.
When $a<b$, the inclusion~\eqref{wws} will follow once we prove $(a,b)\subset V(\partial \mathbb F_\circ)$ since the latter set is compact.

Suppose that $t\in (a,b)$ but $t\notin V(\partial \mathbb F_\circ)$. Then
\[
V^{-1}\{t\} \subset (\inte \mathbb F_\circ)\cup (\C\setminus \mathbb F_\circ)
\]
where $\inte \mathbb F_\circ$ stands for the interior of $\mathbb F_\circ$.
By Claim~B the set $V^{-1}\{t\}$ is a continuum of diameter at least $d$ and therefore cannot be a subset of $\inte \mathbb F_\circ$.
Hence $V^{-1}\{t\}\subset \C\setminus \mathbb F_\circ$, but this contradicts the assumption $t\in (a,b)\subset V(\mathbb F_\circ)$.
Completing  the proof of~\eqref{wws}.

From $\partial \mathbb F_\circ \subset \partial \mathbb F$   it follows that
$V(\mathbb F_\circ)\subset V(\partial \mathbb F_\circ)\subset V(\partial\mathbb F)$. Since $\mathbb F_\circ$ was an arbitrary component of $\mathbb F$,
the lemma is proved.
\end{proof}

We now return to the sequence $\{h_\nu\}\subset \mathcal F_\Y (\X)$ defined in~\eqref{badnu} and the associated functions
\begin{equation}\label{associated}
V_\nu(z)=V_{h_{\nu}}(z) \qquad \text{in }\ C(\C)\cap W_{\rm loc}^{1,2}(\C).
\end{equation}
In view of~\eqref{tristar} and~\eqref{obviously} we have the uniform bound on the Dirichlet integrals
\[
\int_{\C} \abs{\nabla V_\nu(z)}^2\,dz \le \norm{\nabla U}_{\infty} \frac{\diam^2 \Y}{\epsilon^2},
\qquad \nu =1,2,\dots
\]
Since $V_\nu$ have the oscillation property on every disk $\B$ of diameter $d$, the estimate~\eqref{continuity} applies, yielding
\[
\abs{V_\nu(a)-V_\nu(b)}^2 \le \frac{C\norm{\nabla U}_{\infty}\diam^2\Y}{\epsilon^2 \log \big(e+\frac{d}{2\abs{a-b}}\big)}
\]
whenever $a,b\in\C$ and $\abs{a-b}\le \frac{1}{2}d$.

This shows that the functions $V_\nu$ are equicontinuous on $\C$. By the Arzel\`a-Ascoli theorem there is a subsequence, again denoted
$\{V_\nu\}$, that converges uniformly on $\C$ to a continuous function $V=V(z)$. In particular,
\begin{equation}\label{twostar}
V_\nu(z_\nu)-V(z_\nu) \to 0 \qquad \text{as }\ \nu\to\infty.
\end{equation}
Also note that $V_\nu\equiv 0$ on $\X_i$ so $V\equiv 0$ on $\X_i$ as well. On the other hand, it follows from the definition of $V_\nu$
that $V_\nu(z_\nu)=U\big(h_\nu(z_\nu)\big)$, and from~\eqref{badnu} we know that $h_\nu(z_\nu)$ stay away from $\y_i$, precisely
\[
\dist \big( h_\nu (z_\nu), \y_i \big) \ge \epsilon \, \sqrt{\E [h_\nu]} \ge \epsilon \, \sqrt{c (\X, \Y)}.
\]
Hence there is $t_\circ>0$ such that $V_\nu(z_\nu)=U(h_\nu(z_\nu))\ge t_\circ$ for all $\nu=1,2,\dots$. Passing to the limit in~\eqref{twostar} as $z_\nu\to z_\circ \in\x_i$,
we obtain a contradiction
\[
t_\circ = t_\circ -V(z_\circ) \le \lim_{\nu\to\infty} \big[ V_{\nu}(z_\nu) - V(z_\nu) \big] = 0
\]
thus completing the proof of Theorem~\ref{limittheorem}.
\end{proof}

\bibliographystyle{amsplain}

\end{document}